\documentclass[11pt,a4paper]{article}
\usepackage[centertags]{amsmath}
\usepackage{amsfonts}
\usepackage{amssymb}
\usepackage{amsthm}
\usepackage{epsfig}
\usepackage{setspace}
\usepackage{ae} 
\usepackage{eucal} 
\usepackage[usenames]{color}
\usepackage{graphicx}

\usepackage{setspace}
\theoremstyle{plain}
\newtheorem{thm}{Theorem}[section]
\newtheorem{lem}[thm]{Lemma}
\newtheorem{ass}[thm]{Assumption}

\theoremstyle{definition}

\newtheorem{alg}[thm]{Algorithm}

\theoremstyle{remark}

\newtheorem{rem}[thm]{Remark}

\newcommand{\R}{\mathbb{R}}
\newcommand{\N}{\mathbb{N}}
\newcommand{\Z}{\mathbb{Z}}
\newcommand{\F}{\mathcal{F}}
\newcommand{\LL}{\mathcal{L}}
\newcommand{\U}{\mathcal{U}}
\newcommand{\V}{\mathcal{V}}
\newcommand{\D}{\mathcal{D}}
\newcommand{\ii}{{\bf i}}
\newcommand{\jj}{{\bf j}}
\newcommand{\dd}{\textnormal{d}}
\DeclareMathOperator{\Var}{Var}

\DeclareMathOperator{\Span}{span}

\DeclareMathOperator{\supp}{supp}
 \DeclareMathOperator*{\esssup}{esssup}

\DeclareMathOperator*{\arginf}{arginf}

\begin{document}

\title{`Regression Anytime' with Brute-Force SVD Truncation}

\author{Christian Bender$^{1}$ and Nikolaus Schweizer$^2$}

 \maketitle
\footnotetext[1]{Saarland University, Department of Mathematics,
Postfach 151150, D-66041 Saarbr\"ucken, Germany, {\tt
bender@math.uni-sb.de}} \footnotetext[2]{Tilburg University, Department of Econometrics and OR, PO box 90153, NL-5000 LE Tilburg, The Netherlands, {\tt n.f.f.schweizer@uvt.nl}}

\begin{abstract}
We propose a new least-squares Monte Carlo algorithm for the approximation of conditional expectations in the presence of stochastic derivative weights. The algorithm can serve as a building block 
for solving dynamic programming equations, which arise, e.g., in non-linear option pricing problems or in probabilistic discretization schemes for
fully non-linear parabolic partial differential equations.  Our algorithm can be generically 
applied when the underlying dynamics stem from an Euler approximation to a stochastic differential equation. A built-in
variance reduction ensures that the convergence in the number of samples to the true regression function takes place at an arbitrarily fast polynomial 
rate, if the problem under consideration is smooth enough.


\end{abstract}

\section{Introduction}

Approximating conditional expectation functions numerically is one of the central difficulties when solving dynamic programming problems in financial and economic applications \cite{glasserman, powell, rust}, or in implementations
of probabilistic time-discretization schemes for parabolic partial differential equations \cite{BT, Zh, FTW, Ta}. For instance, when solving an optimal stopping problem numerically, there are trade-offs between stopping now for an immediate reward or waiting, receiving the continuation value -- which is a conditional expectation of future rewards.
In a highly influential paper, Longstaff and Schwartz \cite{LS} proposed to compute conditional expectations functions within Monte Carlo simulations in exactly the same way such functions are estimated from real world data. In 
their least-squares Monte Carlo (LSMC) approach, conditional expectations are approximated by regressing future realizations of some quantity of interest on basis functions (e.g. polynomials) that depend 
on current values of the state variables. This approach of mimicking the empiricist's regression method with real data replaced by simulated data has been the starting point of a vast and successful LSMC literature, see \cite{TVR, schwartz, AB, LGW, BGSS, MRHT, FTW, GHL, DSS, KLP, GLTV, NMS} for a broad selection of contributions from various fields ranging from economics and finance to numerical analysis. Yet early on, Glasserman and Yu \cite{GY} pointed out that, in principle, exploiting properties of the Monte Carlo setting that are not available to the empiricist can lead to even more powerful algorithms. Specifically, they proposed to consider LSMC algorithms where the basis functions depend on future values of the state variables. They called this type of algorithm `regression later' to contrast it against traditional LSMC algorithms which rely on a `regression now'.

When the underlying simulation model has the property that conditional expectations of the basis functions can be computed in closed form, `regression later' provides an alternative method for approximating conditional expectations. For instance, when the basis functions are polynomials, regression later requires that conditional moments of the simulated state process are available in closed form. To appreciate the promise of  `regression later', consider a setting where the regression model is correctly specified in the sense that 
the quantity within the time-$t$ conditional expectation is as a linear combination of the basis functions. In that case, under very mild conditions, one observation per basis function  will suffice to determine the time-$t$ conditional expectations function without error using `regression later'. In contrast, even when the model is correctly specified, `regression now' will suffer from the classical (and slow) square-root convergence of Monte Carlo methods.

Despite this tremendous promise, `regression later' has not replaced `regression now' as the default algorithm for numerical approximation of conditional expectations in its first 15 years. In our view, 
there are at least three reasons for this. (i) The method is harder to implement than `regression now' as it requires implementing not
only the basis functions but also their conditionals expectations. On top of this, in `regression later', basis choice is restricted to functions
with known closed-form conditional expectations while it is almost unrestricted in `regression now'. (ii) It took time
to become clear that the most 
successful applications of `regression later' seem to be outside the original optimal stopping setting of \cite{GY}, 
see e.g. \cite{BS,  Bal, PS, BGS2, BP}. 
Crucially, \cite{BS} pointed out that, under mild assumptions, Malliavin derivatives of the quantity of interest can
be computed without additional error 
in `regression later'. This observation is (largely) irrelevant for optimal stopping. Yet it implies that one of the greatest difficulties 
of `regression now', 
namely the variance explosion of the stochatic derivative weights,
 does not exist in the `regression later' setting (see \cite{AA} for discussion and control variates within `regression now'). (iii) Finally, while `regression later'
promises to converge faster than `regression now', there is, despite partial results in \cite{Bal}, a lack of theoretical support for this claim. The main obstacle in the theoretical analysis is that the usual derivations of convergence rates for regression problems rely on truncations to stabilize the approximation. Yet, basically, this truncation would have to be applied after the regression but before the closed-form calculation -- thus destroying the scope for making the calculation in closed-form.

To address these shortcomings and combine the advantages of `regression later' and `regression now', this paper introduces and rigorously analyzes `Regression Anytime with Brute-Force SVD Truncation' (RAWBFST, pronounced ``raw-beef-st''). Under sufficient smoothness, our algorithm can be calibrated to achieve any polynomial convergence rate for the mean squared error, thus holding some of the (bold) promise of convergence with a finite number samples in a fairly generic setting. These rates are achieved not only for the conditional expectations functions but also for their Malliavin derivatives that can be computed simultaneously.

The first ingredient of RAWBFST is what we call `regression anytime'. The idea is to let the basis functions depend on the state variables both `now' and `later'. In particular, we consider basis functions that are products of a function that depends on `now' and a function that depends on `later'. When taking the conditional expectation given `now' of such a function, the first factor depending only on `now' can be pulled out of the expectation. Thus, to compute the conditional expectation of the basis function, it suffices to consider the second factor. `Regression anytime' was previously applied within the stochastic grid bundling method of \cite{JO,CO} and in the LSMC algorithm of \cite{BGS}. While the idea of  `regression anytime' is simple, its additional flexibility in basis choice is key for developing practically effective implementations of `regression later', both in their settings and in ours. `Regression later' is a special case of `regression anytime' and, after a suitable redefinition of the state process, `regression anytime' can be rephrased as `regression later'. In this sense, the step from `regression later' to `regression anytime' is a change of perspective and a reevaluation of possibilities rather than the invention of a new algorithm.

The second ingredient of RAWBFST is a `Brute-Force SVD Truncation'.
Before the regression, we compute the singular values of the empirical regression matrix.
If all singular values are above a previously specified threshold, we perform the regression in the usual way, 
otherwise we set all coefficients to zero. In a more abstract framework, a similar idea appeared previously in \cite{CM} under 
the name `conditioned least-squares approximation'. A key ingredient of our error analysis for RAWBFST is a bound 
on the approximation error of noiseless regression with brute-force SVD truncation, Theorem \ref{thm:noiseless}. Here, the term `noiseless'
refers to settings in which the observations and the explanatory variables are driven by the same randomness so that the 
regression problem is ultimately an interpolation problem.  We show that the statistical error of this type of regression vanishes 
exponentially quickly in the number of Monte Carlo samples for a fixed set of basis functions. Hence,
up to log-factors, we can let the number of
basis functions and of Monte Carlo samples grow at the same rate.
This exponential decay is the reason 
why the RAWBFST algorithm can converge at any given polynomial rate, if the problem is sufficiently smooth, and 
thus may do better than the usual Monte Carlo convergence rate. 

Theorem \ref{thm:noiseless} can be viewed as a generalization of the results in \cite{CDL} 
beyond the case of orthonormal basis functions.   Exploiting the 
matrix version of the Bernstein inequality we explicitly derive, 
how the  exponential decay rate in noiseless least-squares regression depends on the eigenvalues of the expected regression matrix, i.e. the matrix containing 
the $\LL^2$-inner products of the basis functions with respect to the law of the state variable. 
 Thus, to fully exploit  Theorem \ref{thm:noiseless} some control of the eigenvalues  of the expected regression matrix is necessary. 
 The RAWBFST algorithm provides a guaranteed control of these eigenvalues in generic settings where the state process is discretized via (one step of) a Euler scheme. 
 The basis functions of RAWBFST need not be tailored to the coefficients of the Euler scheme, 
 but are generically chosen as Legendre polynomials at the later time point localized to a grid at the earlier time point.
   By changing the distribution of the earlier time point of the Euler step to a stratified uniform distribution via importance sampling, we can control how strongly the localized Legendre polynomials deviate from orthonormality.

The paper is organized as follows: In Section \ref{sec:alg}, we first introduce the setting and the RAWBFST algorithm. 
Theorem \ref{thm:main} then provides the error analysis of RAWBFST. Next, we discuss the resulting computational complexity and compare it to results for `regression now' and `regression later'. The theoretical convergence results are tested in three numerical examples in Section \ref{sec:numerics}.  In Section \ref{sec:interpolation}, we analyze the convergence behavior of noiseless regression with brute-force SVD truncation, thus providing a main building block for the error analysis of the RAWBFST algorithm. Finally, the technical details of this error analysis are provided in Section \ref{sec:proof}.

\subsection*{Notation}
For vectors $x=(x_1,\ldots,x_D)\in \R^D$, we write $|x|_p=(\sum_d |x_d|^p)^{1/p}$ for the $p$-norm $(p\geq 1)$ and  $|x|_\infty=\max_d |x_d|$ for the maximum norm. Given a symmetric  matrix $A\in \R^{D\times D}$, 
$\lambda_{\min}(A)$ and $\lambda_{\max}(A)$ denote the smallest and largest eigenvalue  of $A$ respectively. For general matrices $A\in \R^{D_1\times D_2}$, we apply the spectral norm $\|A\|_2=\sqrt{\lambda_{\max}(AA^\top)}$, 
where $(\cdot)^\top$  stands for matrix transposition. By $O(D)$ we denote the set of orthogonal matrices in $\R^{D\times D}$. We
write $\mathcal{C}^{Q}_b(\R^D)$ for the space of bounded real valued function on $\R^D$, which are $Q$-times continuously differentiable with bounded derivatives ($Q\in \N_0$), 
and $\|f\|_\infty=\sup_{x\in \R^D} |f(x)|_2$ for the sup-norm of a function $f:\R^D\rightarrow \R^M$. 
$\chi^2_D(\alpha)$ denotes the $(1-\alpha)$-quantile of the $\chi^2$-distribution with $D$ degrees of freedom while $\Phi$ and $\varphi$ stand for the distribution function and density of the standard normal distribution, respectively. 
Given a random vector $X$ in $\R^D$, we write $\supp X$ for the support of $X$, i.e. the set of $x\in \R^D$ such that $X$ hits every $\epsilon$-ball around $x$ with positive probability.
For a vector $x\in \R^D$ and a constant $r>0$ we denote by $[x]_r$ the componentwise truncation at level $\pm r$, i.e.
$$
[x]_r=([x_1]_r,\ldots [x_D]_r)^\top,\quad [x_d]_r= \max(\min(x_d,r),-r).
$$

\section{RAWBFST: Algorithm, convergence result, and discussion}\label{sec:alg}

\subsection{Setting of the problem}

Our main motivation is the problem of approximating conditional expectations of the form 
\begin{equation}\label{eq:cond_ex}
 E[\mathcal{H}_{\iota,\Delta}(\xi)  y(X_2)|X_1]
\end{equation}
via empirical least-squares regression, where $X_2$ is one step of an Euler scheme with step size $\Delta$ starting at $X_1$ and  $\mathcal{H}_{\iota,\Delta}(\xi)$
is a Malliavin Monte Carlo weight for the approximation of a (higher order) partial derivative of $y$. These type of conditional expectations appear in discretization schemes for backward stochastic differential equations (BSDEs), see e.g. \cite{Zh,BT}, and, more 
generally, in stochastic time discretization schemes of fully nonlinear parabolic partial differential equations (PDEs), see \cite{FTW, Ta}, including 
Hamilton-Jacobi-Bellman equations arising from stochastic control problems.

More precisely, let $X_1$ be an $\R^D$-valued random variable with law $\mu_1$, and denote by $\xi$ a $D$-dimensional vector of 
independent standard normal random variables, which is also assumed to be independent of $X_1$. For measurable coefficient functions $b:\R^D\rightarrow \R^D$ and $\sigma:\R^D\rightarrow \R^{D\times D}$, we consider 
\begin{equation}\label{eq:onestep}
X_2=X_1+b(X_1)\Delta+\sigma(X_1)\sqrt{\Delta}\xi.
\end{equation}
Conditions on the law of $X_1$ and on the coefficient functions will be specified in Assumption \ref{ass:1}. On the function $y:\R^D\rightarrow \R$ we assume that it is $Q+1$-times continuously differentiable and bounded with bounded derivatives, for some $Q\in \N$.
The boundedness assumptions can, of course, be relaxed, but we impose them for sake of simplicity. In order to specify the stochastic weights, we denote by 
$$
\mathcal{H}_q(x)=(-1)^qe^{x^2/2} \frac{\dd^q}{\dd x^q} e^{-x^2/2},\quad x\in \R,
$$
the Hermite polynomial with parameter 1 of degree $q\in \N_0$. For a multi-index $\iota\in \N_0^D$ , we denote its absolute value by $|\iota|_1=\sum_{d=1}^D \iota_d$. Then, the stochastic weight is defined as a scaled multivariate Hermite polynomial of degree $|\iota|_1$, namely,
$$
\mathcal{H}_{\iota,\Delta}(x)=\Delta^{-|\iota|_1/2} \prod_{d=1}^D  \mathcal{H}_{\iota_d}(x_d),\quad x=(x_1,\ldots,x_D).
$$
Moreover, we write $\bar \iota$ for the vector in $\{1,\ldots D\}^{|\iota|_1}$ which has, for each $d=1,\ldots,D$, the entry $d$ $\iota_d$-times, and whose entries are increasingly ordered.
Note that, under the assumptions stated above,  integration by parts yields, for $1\leq |\iota|_1\leq Q-1$,
\begin{eqnarray}\label{eq:relation_weight_derivative}
 && \nonumber E[\mathcal{H}_{\iota,\Delta}(\xi)  y(X_2)|X_1]\\ &=&\sum_{j_1,\ldots, j_{|\iota|_1}=1}^D \sigma_{j_1,\bar\iota_1}(X_1)\cdots \sigma_{j_{|\iota|_1},\bar\iota_{|\iota|_1}}(X_1) 
E\left[\left. \frac{\partial^{|\iota|_1}}{\partial (x_{j_1},\ldots x_{j_{|\iota|_1}})}y(X_2)\right|X_1\right].
\end{eqnarray}
Hence, the conditional expectation \eqref{eq:cond_ex} approximates the weighted sum of partial derivatives of $y$
$$
\sum_{j_1,\ldots, j_{|\iota|}=1}^D \sigma_{j_1,\bar\iota_1}(X_1)\cdots \sigma_{j_{|\iota|_1},\bar\iota_{|\iota|_1}}(X_1) 
 \frac{\partial^{|\iota|_1}}{\partial (x_{j_1},\ldots x_{j_{|\iota|_1}})}y(X_1),
$$
as $\Delta$ tends to zero. By a first-order Taylor expansion of
$$
\frac{\partial^{|\iota|_1}}{\partial (x_{j_1},\ldots x_{j_{|\iota|_1}})}y(X_2)
$$
around $X_1$, 
 this convergence will be of order $\Delta$ in  $\LL^2(\Omega,\F,P)$, if $b$ and $\sigma$ are bounded.

\subsection{The algorithm and its convergence behavior}

We  now introduce RAWBFST, our new algorithm for the approximation of conditional expectations in the presence of Malliavin Monte Carlo weights,
i.e. of the form \eqref{eq:cond_ex}. It is what we call a `regression anytime'-type algorithm, in the sense that we approximate 
$y(X_2)$ by basis functions which depend on $X_1$ and $X_2$. The basis functions are chosen in a way that the 
conditional expectations of the approximation multiplied by the Malliavin Monte Carlo weight are available in closed form,
leading to some type of automatic differentiation.
The algorithm  also relies on a change of measure of the 
law of $X_1$ and employs stratification, similarly to the  `regression now'-algorithm for backward 
stochastic differential equations in \cite{GLTV}. 
In order to stabilize the empirical regression, we truncate 
the singular value decomposition of the empirical regression matrix.

On the law of $X_1$ and the coefficients of the Euler scheme, we impose the following assumptions:
\begin{ass}\label{ass:1}
 The law $\mu_1$ of $X_1$ has a density $f$ with respect to the Lebesgue measure such that the `Aronson type' estimate
$$
f(x)\leq  \frac{C_{1,f}}{(2\pi C_{2,f})^{D/2}} \exp \left\{\frac{-|x|_2^2}{2 C_{2,f}} \right\} ,\quad x\in \R^D,
$$ 
is satisfied for constants $C_{1,f}, C_{2,f}>0$. Moreover, $b$ and $\sigma$ are bounded, i.e. there is a  constant $C_{b,\sigma}>0$  such that 
$$
\sup_{u\in \R^D} \left( |b(u)|_2+ \|\sigma(u)\|_2\right)\leq C_{b,\sigma}.  
$$
\end{ass}
\begin{rem}
 The situation which we have in mind is the following one:  $X_1=X^{e}_{t_{i_0}}$ for some $i_0$, where 
$$
X^{e}_{t_{i+1}}=X^{e}_{t_i}+\bar b(t_i, X_{t_i}^e)\Delta+  \bar \sigma(t_i, X_{t_i}^e)\sqrt{\Delta}\xi_i,\quad X^{e}_{0}=X_0,
$$
is an Euler scheme approximation to the stochastic differential equation 
$$
dX(t)=\bar b(t,X(t))dt+\bar \sigma(t,X(t))dW(t), \quad X(0)=X_0, \quad t\in [0,T],
$$
driven by a Brownian motion $W$. Here, of course, $t_i=i\Delta$ and $(\xi_i)$ is an i.i.d. family of $D$-dimensional vectors of independent standard normal variables. Suppose $\bar b:[0,T]\times \R^D\rightarrow \R^D$ 
is measurable and bounded,  $\bar\sigma:[0,T]\times \R^{D}\rightarrow \R^{D\times D}$ is measurable, H\"older continuous in space (uniformly in time), and $\bar\sigma \bar\sigma^\top$ is uniformly elliptic. 
Moreover, assume that $X_0$ is  independent of $(\xi_i)$ and Gaussian with mean vector $x_0$ and covariance matrix $\Sigma_0\geq c_0 \mathbb{I}_D$ (where $\mathbb{I}_D$ is the identity matrix and $c_0\geq 0$, i.e. $\Sigma_0$ may be degenerate). Then, by an application of Theorem 2.1 in \cite{LM}, the Aronson estimate in Assumption \ref{ass:1} holds with 
$C_{2,f}=C_{2,f}' t_{i_0}+c_0$ for some constant $C_{2,f}'>0$.
\end{rem}

Before we precisely state the algorithm, let us first explain its several steps in a more intuitive way: In the course of Steps 1--4 of the algorithm,
we construct an approximation $\hat y(X_1,X_2;\Theta)$ of the function
 $y(X_2)$.   The function $\hat y(x_1,x_2;\Theta)$ is a linear combination of polynomials in $x_2$, which are localized in the $x_1$-variables (i.e.  `one time step earlier'). The approximation depends on a randomly generated sample $\Theta=(U_{\ii,l},\xi_{\ii,l})$. The first two steps of the algorithm are preparations. In Step 1, a cubic partition $(\Gamma_\ii)$ of some 
subset $\Gamma\subset \R^D$ for the localization in the $x_1$-variables is constructed. In Step 2, 
we provide a suitable basis of the space of polynomials of degree at most $Q$ in terms of Legendre polynomials. With these polynomials, we define our basis functions of the type `polynomials localized one time step earlier'. 
In Steps 3--4, an empirical regression with SVD truncation is performed to compute the coefficients for $\hat y(x_1,x_2;\Theta)$. Here,  we first change 
measure 
from the true distribution of $X_1$ to the uniform distribution on $\Gamma$ and then stratify the uniform distribution on $\Gamma$ on the cubic partition $(\Gamma_\ii)$. This change to a uniform distribution is in line with our choice of Legendre polynomials for the basis functions as these are the orthogonal polynomials for the uniform distribution. The sampling of $X_2$ in Step 3 involves an additional truncation 
of the Gaussian innovations at some level $r_2$. 
In the final Step 5, the algorithm returns our estimator $\hat{z}(x)$  for $E[\mathcal{H}_{\iota,\Delta}(\xi)  y(X_2)|X_1=x]$. 
 This estimator is simply the closed-form expression for the conditional expectation 
$$
E[\mathcal{H}_{\iota,\Delta}([\xi]_{r_2})  \hat y(X_1,X_2^{(\Delta,r_2)};\Theta)|\Theta, X_1=x],
$$
where $(X_1,\xi)$ is independent of the sample $\Theta$ and $X^{(\Delta,r_2)}_2$ is the one-step Euler scheme \eqref{eq:onestep} starting from $X_1$ with step size $\Delta$ and
with the truncated Gaussian innovation $[\xi]_{r_2}$ in place of $\xi$.

\begin{alg}\label{alg:1}
 {\it Input:} function $y:\R^D\rightarrow \R$, constants $L\in \N$, $Q\in \N$, $\Delta>0$, $\iota\in \N_0^D$, $\tau\in (0,1)$, $\gamma_{\textnormal{cube}}\in (0,1/2)$, $c_{\textnormal{cube}}, c_{1,\textnormal{trunc}}, c_{2,\textnormal{trunc}}, 
\gamma_{1,\textnormal{trunc}}, \gamma_{2,\textnormal{trunc}}>0$.
\begin{itemize}
 \item Step 1: Construction of the cubic partition for $X_1$ (`now'). 

Let $h=c_{\textnormal{cube}}\Delta^{\gamma_{\textnormal{cube}}}$ and $r_1=\sqrt{C_{2,f} \chi^2_D(c_{1,\textnormal{trunc}}\,\Delta^{ \gamma_{1,\textnormal{trunc}}})}$. For every multi-index $\ii=(i_1,\ldots, i_D)\in \Z^D$, consider the cube 
$$
\Gamma_\ii=\prod_{d=1}^D (hi_d,h(i_d+1)] 
$$
and let 
$$
I=I_{\Delta}=\left\{\ii \in \Z^D;\; \Gamma_\ii\cap \{x\in \R^D;\ |x|_2\leq r_1\} \neq \emptyset\right\}.
$$
Write $\Gamma=\cup_{\ii\in I} \Gamma_\ii$ and $a_\ii$ for the center of the $\ii$th cube. 
\item Step 2:   Construction of the local polynomials for $X_2$ (`later').
Denote by $\mathcal{L}_q:\R\rightarrow \R$ the Legendre polynomial of degree $q$, which is normalized such that $\mathcal{L}_q(1)=1$, i.e. 
\begin{equation*}
 \mathcal{L}_q(x)=\frac{1}{2^q}\sum_{r=0}^{\lfloor q/2\rfloor} \frac{(-1)^r (2q-2r)!}{r!(q-r)!(q-2r)!}x^{q-2r}.
\end{equation*}
For any multi-index $\jj\in \N_0^D$ such that $|\jj|_1\leq Q$,  let 
$$
p_{\jj}(x)=  \prod_{d=1}^D \sqrt{2j_d+1} \, \mathcal{L}_{j_d}\left(x_d\right),\quad x=(x_1,\ldots, x_D).
$$
Let $K= {{D+Q}\choose{D}}$. For every $\ii\in I$ denote by $\eta_{\ii,k}$, $k=1,\ldots,K$, any fixed ordering of the polynomials
$$
x\mapsto  p_{\jj}\left(\frac{x-a_\ii}{h/2}\right),\quad { \jj}\in \N_0^D,\, |\jj|_1\leq Q.
$$ 

\item Step 3: Construction of the empirical regression matrices on the cubes.

For every $\ii\in I$, sample independent copies $(U_{\ii,l},\xi_{\ii,l})_{l=1,\ldots,L}$ where $U_{\ii,l}$ is uniformly distributed on $\Gamma_\ii$ and   $\xi_{\ii,l}$ is multivariate Gaussian with zero mean vector and unit covariance matrix independent of $U_{i,l}$. Let 
$r_2=\sqrt{2\log(c_{2,\textnormal{trunc}}\,\Delta^{- \gamma_{2,\textnormal{trunc}}}\log(\Delta^{-1}))}$ and
$$
X_{\ii,l}=U_{\ii,l}+b(U_{\ii,l})\Delta+\sigma(U_{\ii,l})\sqrt{\Delta} \,[\xi_{\ii,l}]_{r_2}.
$$
Build the empirical regression matrices
$$
A_\ii=(\eta_{\ii,k}(X_{\ii,l}))_{l=1,\ldots,L;\,k=1,\ldots K}
$$
\item Step 4: Least-squares interpolation with brute-force SVD truncation.

For every $\ii\in I$ perform a singular value decomposition of $A_\ii^\top$:
$$
A_\ii^\top=\mathcal{U} \mathcal{D}\mathcal{V},\quad \mathcal{U}\in O(K),\; \mathcal{V}\in O(L),
$$
where $\mathcal{D}$ is the $K\times L$-matrix which has the singular values $s_1\geq s_2\geq \cdots\geq s_K\geq 0$ of $A_\ii$ on the diagonal and has zero entries otherwise.
If $s_K^2\geq \tau L$, let
$$
\alpha_{L,\ii}=\mathcal{U}\mathcal{D}^\dagger \mathcal{V} \;(y(X_{\ii,1}),\ldots, y(X_{\ii,L}))^\top
$$
where $\mathcal{D}^\dagger$  is the $K\times L$-matrix which has $s_1^{-1},\ldots,s_K^{-1}$  on the diagonal and has zero entries otherwise (i.e., the pseudoinverse of $\mathcal{D}^\top$).
Otherwise let
$$
\alpha_{L,\ii}=0\in \R^K.
$$
\item Step 5: Return 
$$
\hat{z}(x):=\sum_{\ii\in I} {\bf 1}_{\Gamma_\ii}(x) \sum_{k=1}^K \alpha_{L,\ii,k} E[\eta_{\ii,k}(x+b(x)\Delta+\sigma(x)\sqrt{\Delta} [\xi]_{r_2}) \mathcal{H}_{\iota,\Delta}([\xi]_{r_2})]
$$
as an approximation of $E[\mathcal{H}_{\iota,\Delta}(\xi)  y(X_2)|X_1=x]$.
\end{itemize}
\end{alg}

\begin{rem}
 Note that 
$$
\eta_{\ii,k}(x+b(x)\Delta+\sigma(x)\sqrt{\Delta} [\xi]_{r_2}) \mathcal{H}_{\iota,\Delta}([\xi]_{r_2})
$$
is a polynomial in $[\xi_1]_{r_2},\ldots [\xi_D]_{r_2}$, whose coefficients depend on $x$. By independence, 
$$
E\left[\prod_{d=1}^D [\xi_d]^{q_d}_{r_2} \right]=\prod_{d=1}^D  E\left[[\xi_1]^{q_d}_{r_2} \right].
$$
Hence, the expectation in Step 5 can be computed in closed form by the following recursion formula for the moments $m_{q,r}=E[[\xi_1]^{q}_{r}]$  of the truncated standard normal distribution:
\begin{eqnarray*}
 m_{2q,r}&=&(2q-1) m_{2q-2,r} + 2 r^{2q-2}(r^2-2q+1) (1-\Phi(r))- 2 r^{2q-1} \varphi(r),\quad   m_{0,r}=0,\\
m_{2q-1,r}&=&0.
\end{eqnarray*}
\end{rem}

The following theorem provides the error analysis for Algorithm \ref{alg:1}. Its proof is postponed to Section \ref{sec:proof}.

\begin{thm}\label{thm:main}
 Fix $\rho\in \N$. Suppose $y\in \mathcal{C}^{Q+1}_b(\R^D)$ for some $Q\geq |\iota|_1+\rho$. Compute $\hat{z}$ via Algorithm \ref{alg:1} with 
 $$
\gamma_{\textnormal{cube}}=\frac{\rho+|\iota|_1}{2(Q+1)},\quad  \gamma_{1,\textnormal{trunc}}=\rho, \quad\gamma_{2,\textnormal{trunc}}=1.5(|\iota|_1+\rho),$$
and
\begin{eqnarray*}
L&=&L_\Delta=\lceil \rho  \,c_{1,\textnormal{paths}} \log(c_{2,\textnormal{paths}}\; \Delta^{-1})\rceil \\
\tau&\in& \left(0, \quad 1-\left(\frac{c^*_{\textnormal{paths}}(Q,D)}{c_{1,\textnormal{paths}}}\right)^{1/2} \right) 
\end{eqnarray*}
for constants
 $$
c_{2,\textnormal{paths}}>0,\quad c_{1,\textnormal{paths}}> c^*_{\textnormal{paths}}(Q,D):=\frac{2}{3}+\frac{8}{3}\sum_{{\bf j}\in \N_0^D; |{\bf j}|_1\leq Q}\; \prod_{d=1}^D (2j_d+1).
$$
Then there are constants $C>0$ and $\Delta_0>0$ (depending on all the constants, including $D$, $Q$, $|\iota|_1$, $\tau$, and the $\mathcal{C}^{Q+1}_b$-norm of $y$)  such that for every $\Delta\leq \Delta_0$
\begin{eqnarray*}
E\left[\int_{\R^D}  |E[\mathcal{H}_{\iota,\Delta}(\xi)  y(X_2)|X_1=x]-\hat{z}(x)|^2\mu_1(dx) \right]\leq C \log(\Delta^{-1})^{D/2} \Delta^\rho.
\end{eqnarray*}
\end{thm}

\begin{rem}
 In view of Theorem \ref{thm:noiseless} below, the choice parameter $\tau$ for the level of the SVD truncation provides a trade-off between 
 the contribution of the projection error and of the statistical error to the overall error analysis. 
 A larger $\tau$ scales down the constant in front of the projection error. However, a large $\tau$ can only be achieved
 by increasing $c_{1,\textnormal{paths}}$ and, thus by increasing the number of simulations which are required to control the statistical 
 error. In our numerical test cases, we choose $\tau$ of the order $0.02$ in order to keep the number of simulations small without 
 observing any negative effect concerning the behavior of the projection error.
\end{rem}

\subsection{Discussion}

In this section, we discuss the scope and limitations of our algorithm and compare it to 
`regression now' and `regression later'-algorithms for the same problem.

\subsubsection{Scope and limitations}

 As our RAWBFST algorithm  applies local basis functions, it can be 
realistically implemented for the case of low and moderate dimensions $D$ only, say up to around $D=10$.
\\[0.2cm]
{\it Complexity:} We assume the setting of Theorem \ref{thm:main}. The main steps of the algorithm are implemented as a loop over the number of cubes. The truncation level $r_1$ 
for the space partition, which is defined in terms of the quantile function of a $\chi^2$-distribution, grows as $\sqrt{\log(\Delta^{-1})}$,
see the proof of Lemma \ref{lem:trunc2} below. Hence, the number of cubes behaves as 
$$
|I_{\Delta}|\sim \Delta^{-D\gamma_{\textnormal{cube}}} (\log(\Delta^{-1}))^{D/2}.
$$
For each cube within this loop, $L_{\Delta}$ random samples are generated (Step 3) and a (thin) singular value decomposition of a $L_{\Delta} \times 
{D+Q \choose D}$-matrix is performed, where  ${D+Q \choose D}$ is the number of basis functions per cube (Step 4). The costs for these 
two steps grow linearly in the number of samples $L_{\Delta}$ and, thus, logarithmically in $\Delta^{-1}$.  Hence, the overall cost 
for computing the full coefficient vector $(\alpha_{L,\ii})$, and thus to construct the approximating function $\hat z$, is of the order
$$
\Delta^{-D\gamma_{\textnormal{cube}}} (\log(\Delta^{-1}))^{1+D/2}
$$
and, in view of Theorem \ref{thm:main}, yields an accuracy (root mean-squared error)  of the order
$$
\log(\Delta^{-1})^{D/4} \Delta^{\rho/2}.
$$
Ignoring the log-factors the complexity for obtaining an accuracy of the order $\epsilon$
is 
\begin{equation}\label{Cx}
C(\epsilon)=\epsilon^{-\frac{2D}{\rho} \gamma_{\textnormal{cube}}} =  \epsilon^{-\frac{D}{Q+1}\left(1+\frac{|\iota|_1}{\rho}\right)} 
\end{equation}
In particular, RAWBFST can beat the plain Monte Carlo complexity $\epsilon^{-2}$ for approximating a single expectation, if 
the dimension-to-smoothness ratio $D/(Q+1)$ is sufficiently small.
\\[0.2cm]
{\it Evaluation costs:}  In the context of dynamic programming equations (see e.g. the numerical examples below), the approximation
$\hat{z}(x)$ is typically evaluated at a new random sample of size $|I_{\Delta}|L_{\Delta}$. This corresponds to the total number 
of samples (i.e. number of cubes times number of samples per cube). As $\hat z$ is a local function, the dominating cost for evaluating 
$\hat z$ at a single point in Step 5 of the algorithm is to determine the cube in which the point lies. This cost grows logarithmically
in the number of cubes. Hence, ignoring again the log-factors, the evaluation costs for a sample of size $|I_{\Delta}|L_{\Delta}$ is of the order
$$
\Delta^{-D\gamma_{\textnormal{cube}}}= \epsilon^{-\frac{D(\rho+|\iota|_1)}{2(Q+1)}}
$$
and thus matches the cost for computing the full coefficient vector $(\alpha_{L,\ii})$.
\\[0.2cm]
{\it Practical applicability:}
In order to get some insight on the practical applicability of the RAWBFST algorithm, we now fix $\rho=2$, which matches the rate 
of the `discretization error' of the Malliavin Monte-Carlo approximation to the partial derivatives of $y$, cp. \eqref{eq:relation_weight_derivative}. Table \ref{tab:raw} provides,
for various choices of the dimension $D$ and the maximal degree of the local polynomials $Q$, the following numbers:
\begin{itemize}
 \item   $\lceil 2  c^*_{\textnormal{paths}}(Q,D)\rceil$, which up to the log-factor corresponds to the number of samples generated per cube;
 \item $D+Q\choose D$, the number of local basis functions per cube;
 \item rate$_{\partial}$, the convergence rate for the approximation of a first order partial derivative as a function of complexity;
 \item rate$_{\partial^2}$, the convergence rate for the approximation of a second order partial derivative (for $Q \geq 4$) as a function of complexity.
\end{itemize}
A cross indicates that we consider the computational cost prohibitive, unless a computer architecture with a massive potential for 
parallelization is at hand.
\begin{table}
\centering
\begin{tabular}{|c||c|c|c|c|c|c|}
\hline
\vspace{-8pt} &&&&&&\\
$Q \backslash D $ & 1& 3 & 5& 7 & 10 & 15\\
\hline  \hline & 87  & 1015 & 3927 & 9975 & 27447 & 88567 \\ 3  & 4  & 20  & 56 & 120 & 286 & 816 \\  &  {\bf 2.67} & {\bf 0.89} & {\bf 0.53} & {\bf 0.38}  & {\bf 0.27} & \bf{0.18} \\  & --  & -- & -- & -- & -- & -- \\
\hline  & 135 &  2951 &  17100 & 58135 & & \\ 4  & 5  & 35  & 126 & 330 & X & X \\  & {\bf 3.33}  & {\bf 1.11}  & {\bf 0.67} & {\bf 0.48} &  & \\  & {\it 2.5}  & {\it 0.83} & {\it 0.5} & {\it 0.36} & & \\
\hline  &  194 & 7319 & 61783 & & & $\lceil 2  c^*_{\textnormal{paths}}(Q,D)\rceil$ \\ 5  & 6 & 56 & 126 &X & X & $D+Q\choose D$ \\  & {\bf 4} & {\bf 1.33} & {\bf 0.8} & & & rate$_\partial$ \\  & {\it 3} & {\it 1} & {\it 0.6}  & & &  rate$_{\partial^2}$  \\
\hline
\end{tabular}
\caption{Summary statistics related to the computational cost and the convergence rates  of the RAWBFST approximation of first order (boldface) and second order (italic) partial derivatives derived from \eqref{Cx} 
for various choices of the dimension $D$ 
and the maximal polynomial degree $Q$.}\label{tab:raw}
\end{table}

\subsubsection{Comparison to `regression now'} \label{sec:RNRLRA}

In the `regression now' approach, one chooses a set of basis functions 
 $$\eta(x)=(\eta_1(x),\ldots,\eta_K(x)),$$ which we think of as a row vector. One then generates independent samples $(X_{1,l},\xi_{l})$, ${l=1,\ldots,L}$, 
 and solves the linear least-squares problem
 \begin{equation}\label{reg:now}
 \hat \alpha^{L,\,\textnormal{now}}=\arginf_{\alpha \in \R^K} \frac{1}{L} \sum_{l=1}^L \left| \mathcal{H}_{\iota,\Delta}(\xi_l)y(X_{2,l})- \eta(X_{1,l})\alpha\right|^2 ,\quad 
 X_{2,l}:=X_{1,l}+b(X_{1,l})\Delta+\sigma(X_{1,l})\sqrt{\Delta}\xi_l.
 \end{equation}
(In the case of multiple minimizers, one can choose e.g. the one with the minimal Euclidean norm). One then approximates the
regression function $z(x)=E[\mathcal{H}_{\iota,\Delta}(\xi)  y(X_2)|X_1=x]$ by
$$
\hat z^{L,\,\textnormal{now}}(x)=[\eta(x)\hat \alpha^{L,\,\textnormal{now}}]_{B},
$$
where truncation takes place at a level $B$, which is any upper bound for the supremum norm of the regression function $z$.
The resulting estimate is, thus, a linear combination of the basis functions, truncated at level $B$. This truncation is the standard way to come up with a `stable' estimate in situations where 
(say, due to an unfavorable realization of the sample) the least-squares regression problem in \eqref{reg:now} is ill-conditioned.  
If one thinks of $X_1$ and $X_2$ as  modeling a system at two time points 1 (`now') and 2 (`later'), the phrase `regression now' 
simply emphasizes that the basis 
functions only depend on $X_1$.

According to Theorem 11.3 in \cite{Gal}, the $\LL^2$-error for this `regression now'-estimate decomposes into the sum of a `projection error' and a `statistical error', which are of the form
\begin{eqnarray}
\textnormal{projection error}&=& \sqrt{\inf_{\alpha \in \R^K} E[|z(X_1)-\eta(X_1)\alpha|^2 ]} \label{eq:proj} \\
\textnormal{statistical error}&=& \sqrt{(\sup_{x\in \R^D} \Var(Z|X_1=x)+B^2)\, \frac{(\log(L)+1)K}{L}},
\end{eqnarray}
where
$$
Z=\mathcal{H}_{\iota,\Delta}(\xi)  y(X_2).
$$
is the regressand. 

As basis functions we apply local polynomials of the form 
\begin{equation}\label{eq:locpoly}
{\bf 1}_{\Gamma_{\bf i}}(X_2) X_2^{\bf j},\quad {\bf j}=(j_1,\ldots,j_D),\; |{\bf j}|_1\leq Q.
\end{equation}
Assuming that $y\in \mathcal{C}^{Q+|\iota|_1+1}$ (and, thus, admitting some extra smoothness compared to Theorem \ref{thm:main}),
a Taylor expansion up to degree $Q$ of the partial derivatives of $y$ of order $|\iota|_1$ shows that 
the number of cubes must increase (up to a log-factor) as 
$
\Delta^{-\frac{D}{Q+1}}
$
to achieve a projection error of the order $\Delta$.

We now turn to the statistical error. Note that the variance of the weights $\mathcal{H}_{\iota,\Delta}(\xi)$ explodes as $\Delta^{-|\iota|_1}$ for $\Delta 
\rightarrow 0$. For $|\iota|_1=1$, this variance blow-up can be prevented in the current setting by replacing $Z$ with 
$$
\tilde Z=\mathcal{H}_{\iota,\Delta}(\xi)  (y(X_2)-y(X_1)), 
$$
since 
$$
E[|\frac{\xi}{\sqrt{\Delta}} (y(X_2)-y(X_1))|_2^2]=E[|\xi \langle \nabla y(X_1),\sigma(X_1)\xi \rangle|^2_2]+O(\Delta^{1/2})=O(1).
$$
Similarly, in the presence of Malliavin weights for higher order derivatives, control variates based on approximations of the lower order partial derivatives 
of $y$ can be pre-computed, see e.g. \cite{AA} for a heuristic argument and \cite{BZG} for a theoretical analysis of this `preliminary control variates' technique. For our comparison, we therefore assume that control variates can be pre-computed 
at negligible computational costs such that $Z$ can be replaced by a random variable with bounded variance. Then, we require (up to a log-factor) 
$
L^{\textnormal{now}}_\Delta \sim \Delta^{-\frac{D}{Q+1}+2}
$
samples to obtain a mean-squared error of the order $\Delta$ (independently of the order of the partial derivative which is to be approximated). Taking into account that each sample can be sorted into the cubes at logarithmic 
cost and that the regression matrix has block-diagonal form due to the local structure of the basis functions, the overall cost is 
(up to a log-term) proportional to the number of samples $L^{\textnormal{now}}_\Delta$. Hence, the `regression now'-algorithm converges 
at a rate of  $(Q+1)/(D+2(Q+1))$  as a function of complexity, if suitable control variates can be constructed. This rate approaches the Monte-Carlo rate of $1/2$, as the 
dimension-to-smoothness ratio decreases to zero. 

We note that the pre-computation of suitable control variates may be nontrivial in practice.
Without the use of control variates, the required  simulation cost behaves as 
$
\Delta^{-\frac{D}{Q+1}+2+|\iota|_1}
$
leading to the slower convergence rate of  $(Q+1)/(D+(2+|\iota|_1)(Q+1))$ as a function of complexity. For comparison to the RAWBFST algorithm, we state in Table \ref{tab:now} the convergence 
rates for the `regression now'-algorithm with control variates for the same choices of $D$ and $Q$ as in Table \ref{tab:raw} for RAWBFST.

\begin{table}
\centering
\begin{tabular}{|c||c|c|c|c|c|c|}
\hline
\vspace{-8pt} &&&&&&\\
$Q \backslash D $ & 1& 3 & 5& 7 & 10 & 15\\
\hline  \hline  3  & 0.44 & 0.36 & 0.31 &0.27 & 0.22 & 0.17 \\  
\hline  4 & 0.45  & 0.38  & 0.33  & 0.29  & X & X \\
\hline  5  & 0.46 & 0.40 & 0.35 &X & X & X \\
\hline
\end{tabular}
\caption{Convergence rates of the `regression now'-algorithm with control variates for various choices of the dimension $D$ and 
the maximal polynomial degree $Q$.}\label{tab:now}
\end{table}
We finally note that the evaluation of the `regression now'-approximation at a single point can be achieved at logarithmic cost,
as the local structure of the basis functions is inherited by the approximation $\hat z^{L,\,\textnormal{now}}$. Hence, the evaluation cost at a new random sample of size 
$L^{\textnormal{reg}}_\Delta$ causes the same cost as the construction of the regression coefficients does. 

\subsubsection{Comparison to `regression later'}
 
The `regression later' approach was suggested in \cite{GY} in the context of optimal stopping, and was later on
applied to backward stochastic differential equations in \cite{BS} under the name `martingale basis method', to insurance liability modeling \cite{PS}, and to discrete time stochastic control
in \cite{BP}, among others. To the best of our knowledge, the first `regression later' algorithm with stochastic derivative weights is due to \cite{BS}. 
In contrast to `regression now', the basis functions in the `regression later' approach depend on $X_2$ and not on $X_1$. 

Precisely, one again chooses a set of basis functions 
 $\eta(x)=(\eta_1(x),\ldots,\eta_K(x))$,  and generates independent samples $(X_{1,l},\xi_{l})$, ${l=1,\ldots,L}$ (which are assumed to be independent of $(X_1,X_2)$).
 After solving the linear least-squares problem
 \begin{equation}\label{reg:later}
 \hat \alpha^{L,\,\textnormal{later}}=\arginf_{\alpha \in \R^K} \frac{1}{L} \sum_{l=1}^L \left| y(X_{2,l})- \eta(X_{2,l})\alpha\right|^2, \quad X_{2,l}:=X_{1,l}+b(X_{1,l})\Delta+\sigma(X_{1,l})\sqrt{\Delta}\xi_l,   
 \end{equation}
one thinks of $\eta(X_2)\hat \alpha^{L,\,\textnormal{later}}$ as an approximation of $y(X_2)$ and hence approximates the regression function $z$ by
\begin{equation}\label{eq:m_later}
\hat z^{L,\,\textnormal{later}}(x)=E\left[\left.\mathcal{H}_{\iota,\Delta}(\xi)  \eta(X_2)\hat \alpha^{L,\,\textnormal{later}}\right|(X_{1,l},\xi_{l})_{l=1,\ldots,L}, X_1=x \right]= \tilde \eta(x)\hat \alpha^{L,\,\textnormal{now}},
\end{equation}
where the entries of $\tilde \eta$ are given by 
\begin{equation}\label{eq:basis_closedform}
\tilde \eta_k(x):= E\left[\left. \mathcal{H}_{\iota,\Delta}(\xi)  \eta_k(X_2) \right|X_1=x\right].
\end{equation}
Note that the linear structure of the regression estimate  $\eta(x)\hat \alpha^{L,\,\textnormal{later}}$ of $y(x)$ is crucial for the closed-form computation of the conditional expectations in \eqref{eq:m_later}. Hence,
in contrast to the `regression now' case, no a-posteriori truncation   of the form $[\eta(x)\hat \alpha^{L,\,\textnormal{later}}]_B$ can be applied to stabilize the empirical regression,
which is why we utilize the SVD truncation in the RAWBFST algorithm.

Noiseless regression problems as in \eqref{reg:later} have been studied e.g. 
in
 \cite{Bal}, \cite{CDL}, \cite{CM}, and the references therein, under the assumption that the basis functions are orthonormal with respect to the law of $X_2$. In the applications, we have in mind, the presence 
 of the Malliavin Monte Carlo weights further imposes the severe restriction that the conditional expectations of the form 
 \eqref{eq:basis_closedform} must be available in closed form. If we again wish to apply local polynomials as basis functions as in \eqref{eq:locpoly},
 this restriction may be hard to meet beyond the Brownian motion case ($b=0$, $\sigma=\mathbb{I}_D$), which is detailed in Section 4 of \cite{Bal}.
 A second drawback of the `regression later' approach is that the conditional expectations in \eqref{eq:basis_closedform} for local 
 polynomials typically lead to functions $\tilde \eta_k(x)$ with global support. Hence, evaluating the corresponding `regression later' approximation 
 to a partial derivative at single point induces a cost proportional to the number of cubes, compared to the logarithmic costs for the RAWBFST
 algorithm or the `regression now' algorithm. This high evaluation cost of the `regression later' algorithm may be considered prohibitive in multidimensional dynamic programming 
 applications like the one discussed in Section \ref{sec:appBSDE} below.

\section{Numerical Illustrations}\label{sec:numerics}

In this section, we provide three numerical illustrations of Algorithm \ref{alg:1}. In the first illustration, we approximate the second derivative of a univariate function. This is a direct application of Theorem \ref{thm:main}. The second illustration is option pricing in the uncertain volatility model, a challenging reference problem in the literature on second-order BSDEs and fully non-linear partial differential equations \cite{GHL,AA,KLP}. The third illustration is the solution of a five-dimensional first-order BSDE considered in \cite{GLTV}. The latter two illustrations are exploratory studies that confirm the excellent performance of our algorithm beyond the setting that is strictly covered by our theoretical analysis. These problems are typical of the applications we envision for our algorithm, as many layers of conditional expectations need to be iterated and we need highly accurate approximations of functions together with their derivatives. 

\subsection{Approximating a second derivative}\label{sec:app2ndD}

We wish to apply Algorithm  \ref{alg:1} to approximate 
\begin{equation}
  z(x)=E\left[\left.\frac{\xi^2 -1}{\Delta}y(X_1+\sqrt{\Delta}\xi) \right|X_1=x\right],\quad x\in \R,
 \end{equation}
where $X_1$  and $\xi$ are independent and standard normal. This setting corresponds to \eqref{eq:cond_ex} with $D=1$, $\iota=2$, $b=0$, and $\sigma=1$. For the function $y$, we consider $y(x)=x^2 \exp(-x^2/2)$. The function $z$ thus approximates the second derivative $y''$ of $y$ with $y''(x)=z(x) +O(\Delta)$. Specifically, we can benchmark the output of our algorithm against the closed-form expressions $y''(x)=(x^4-5x^2+2) \exp(-x^2/2)$ and
\[
z(x) = \frac{x^4-(5+4\Delta-\Delta^2)x^2+2+3\Delta-\Delta^3 }{(1+\Delta)^{\frac92}} \exp\left(-\frac{x^2}{2(1+\Delta)}\right).
\]
In line with Theorem \ref{thm:main}, our main error criterion is the root mean squared error
\[
\mathcal{E}(\Delta,\rho):=\hat{E}\left[\int_{\R}  |z(x)-\hat{z}(x|\Theta,\Delta,\rho)|^2\mu_1(dx) \right]^\frac12
\]
where $\hat{z}(\cdot|\Theta,\Delta,\rho)$ denotes the output of one run of a Julia implementation of Algorithm  \ref{alg:1} in dependence on the Monte Carlo sample $\Theta$, the step size parameter $\Delta$ and the convergence rate parameter $\rho$. $\hat{E}$ denotes an empirical average over 100 runs of the algorithm, i.e., over 100 independent realizations of $\Theta$. $\mu_1$ is the standard normal distribution of $X_1$. The interior univariate integral over $x$ is computed using adaptive quadrature as implemented in Julia's \texttt{quadgk} command.  As a reference, we also compute the discretization error between $y''$ and $z$, 
$$
\bar{\mathcal{E}}(\Delta):=\left(\int_{\R}  |z(x)-y''(x)|^2\mu_1(dx) \right)^{\frac12}.
$$

In our numerical experiments, we vary $\rho=2,3,4$ and $\Delta=2^{-n}$, $n=3,\ldots 14$. The standard normal distribution for $X_1$ implies $C_{1,f}=C_{2,f}=1$. The polynomial degree $Q$, we set as $Q=\rho+\iota+1=\rho+3$. In dimension $D=1$, a direct computation gives
$c^*_{\textnormal{paths}}(Q,1)=\frac{2}{3}+\frac{8}{3}(Q+1)^2$. Accordingly, we choose $c_{1,\textnormal{paths}}= 1.1\, c^*_{\textnormal{paths}}(Q,1)$ and $c_{2,\textnormal{paths}}=1$. The parameters $\gamma_{\textnormal{cube}}$, $\gamma_{1,\textnormal{trunc}}$, $\gamma_{2,\textnormal{trunc}}$ and $L$ are then simply chosen using the formulas given in Theorem \ref{thm:main} while $\tau$ is chosen as the midpoint of the admissible interval given there which implies $\tau = (1-1.1^{-0.5})/2 = 0.0233$. The three remaining parameters that scale the truncation levels and the density of cubes we set as $c_{\textnormal{cube}}=c_{1,\textnormal{trunc}}=c_{2,\textnormal{trunc}}=5$. 

The three black curves in Figure \ref{fig1} plot $\log_{10}(\mathcal{E}(\Delta,\rho))$ against $\log_{10}(\Delta^{-1})$ for $\rho=2,3,4$. Each line is contrasted against a gray line through the final data point with slope equal to the theoretical convergence rate of $\rho/2$ guaranteed by Theorem \ref{thm:main}. We observe that the empirical decay is broadly in line with theory but slightly faster. For comparison, the dotted gray line depicts the discretization error $\log_{10}(\bar{\mathcal{E}}(\Delta))$ which vanishes at a rate close to 1 as expected. Consequently, for $\rho=2$ the approximation error $\mathcal{E}(\Delta,\rho)$ of our algorithm is of a similar magnitude as the discretization error $\bar{\mathcal{E}}(\Delta)$ for all considered values of $\Delta^{-1}$.

\begin{figure}
\centering
\includegraphics[height=240pt]{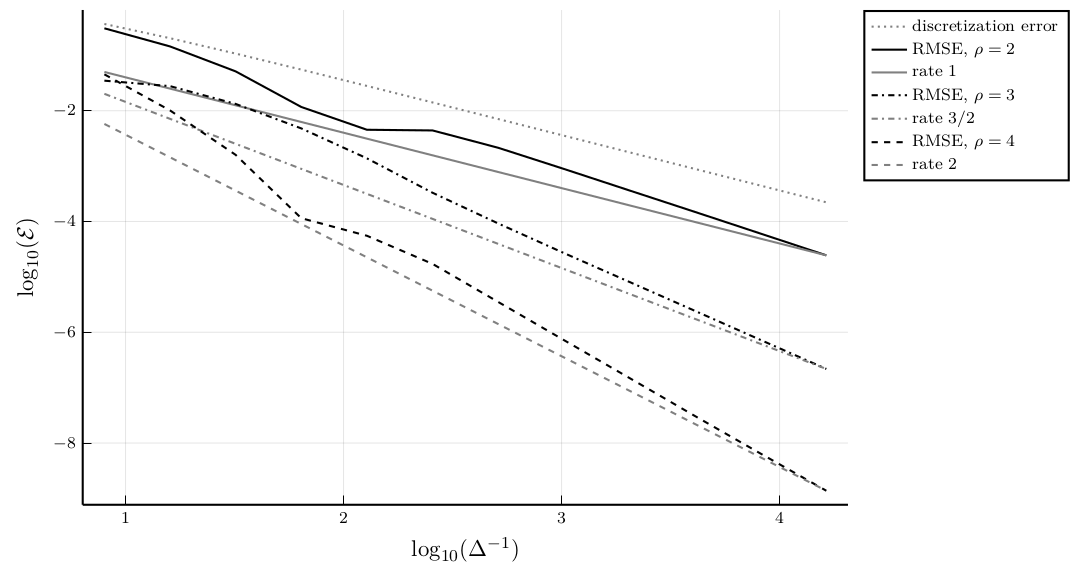}
\caption{Approximation errors against $\Delta^{-1}$ in a $\log_{10}$-$\log_{10}$-plot. The black curves correspond to $\mathcal{E}(\cdot,2)$ (solid line), $\mathcal{E}(\cdot,3)$ (dash-dotted line) and $\mathcal{E}(\cdot,4)$ (dashed line). The corresponding gray lines illustrate the theoretical slopes of 1, 1.5 and 2. The dotted gray line depicts the discretization error $\bar{\mathcal{E}}(\cdot)$.} 
\label{fig1}
\end{figure}

For selected values of $\Delta^{-1}$, Table \ref{tab1} provides further details like number of cubes, number of samples per cube and run times.  For sufficiently small $\Delta$, run times should behave (up to log-factors) like $\Delta^{-\gamma_{\textnormal{cube}}}$. In our implementation, we have $\gamma_{\textnormal{cube}}=\frac{\rho+2}{2\rho+8}$ and essentially the number of cubes grows like $\Delta^{-\gamma_{\textnormal{cube}}}$
while the number of samples per cube only depends on $\Delta$ logarithmically. Thus, thinking optimistically, increasing $\Delta^{-1}$ by a factor 8 should increase the total number of samples and
thus run time by factors of 2, 2.1 and 2.2 for $\rho=2,3,4$. Inspecting the actual run times in the table shows that this type of reasoning is too optimistic in our situation as it ignores logarithmic
factors and rounding effects. Comparing, e.g., the columns associated with $\Delta^{-1}=1024$ and $\Delta^{-1}=8192$, we see that the numbers of cubes $|I_\Delta|$ increase by factors 2.2, 2.5, and 2.52 rather than 
2, 2.1 and 2.2. Similarly, due to logarithmic growth, the number of samples per cube $L_\Delta$ is far from constant, increasing by a factor of about 1.3. Thus, the total number of samples $L_\Delta |I_\Delta|$ increases by factors between 2.86 and 3.28 as $\Delta^{-1}$ is increased by a factor 8. As we implement the thin SVD,  see Remark \ref{rem:svd}, this total number of samples should behave like run time. Indeed, the relative increases in run times we observe throughout the table are broadly consistent with those in the number of samples (but slightly smaller).

\begin{table}
\centering
\begin{tabular}{ccccc}
\hline
\vspace{-8pt} &&&&\\
$\Delta^{-1}$ & 16& 128 & 1024 & 8192\\ 
\hline
\vspace{-8pt} &&&&\\
$\bar{\mathcal{E}}$ & 0.2037 & 0.0280 & 0.0035 & 0.0004\\
\hline
\hline
$\rho=2$ &&&&\\
\hline
\vspace{-8pt} &&&&\\
$\mathcal{E}$ & 0.1447 & 0.0045 & 0.0009  & $6.03\cdot 10^{-5}$\\
$|I_\Delta|$ & 4 & 8 &20&44\\
$L_\Delta$ &590 &1032 &1475  &1917 \\
run time in $s$ & 0.0042& 0.0106& 0.0355& 0.0966\\
\hline
\hline
$\rho=3$ &&&&\\
\hline
\vspace{-8pt} &&&&\\
$\mathcal{E}$ & 0.0278 &$0.0014$& $2.70\cdot 10^{-5}$ & $7.25\cdot 10^{-7}$\\
$|I_\Delta|$ & 4 & 12 &28&70\\
$L_\Delta$ &1202 &2103 &3005  &3906 \\
run time in $s$ & 0.0084& 0.0370& 0.1067& 0.3201\\
\hline
\hline
$\rho=4$ &&&&\\
\hline
\vspace{-8pt} &&&&\\
$\mathcal{E}$ & 0.0101& $5.52\cdot 10^{-5}$& $7.23\cdot 10^{-7}$ & $6.43\cdot 10^{-9}$\\
$|I_\Delta|$ & 6 & 14 &38&96\\
$L_\Delta$ &2091 &3658 &5226  &6794 \\
run time in $s$ & 0.02144& 0.06878& 0.2566& 0.7637\\
\hline
\end{tabular}
\caption{Approximation errors and algorithmic parameters in dependence on $\Delta$ and $\rho$. Run times are for a Julia 1.4.2 implementation on a Windows desktop PC with an Intel Core i7-6700 CPU with 3.4GHz.}\label{tab1}
\end{table}

\subsection{Uncertain Volatility Model}\label{sec:appUVM}
In this section, we apply Algorithm \ref{alg:1} to approximate option prices in the Uncertain Volatility Model (UVM) due to \cite{ALP95,L95}. In this model, it is assumed that the volatility process of the stock underlying an option contract is not known for certain but only known to lie in an interval $[\sigma_{l},\sigma_{h}]$. Option prices are then computed as suprema over all admissible volatility processes in the interval. With this modification, the linear Black-Scholes partial differential equation that arises under a known, constant volatility is replaced by the fully non-linear Black-Scholes-Barenblatt equation. This makes pricing in the UVM a challenging problem even in low dimensions that is frequently used to test new algorithms \cite{GHL,AA,KLP,BGS2}.

In this paper, we directly introduce the discrete time, non-linear backward recursion for pricing that arises when the discretization scheme of \cite{FTW} is applied to the UVM. We refer, e.g., to \cite{BGS2} for a more detailed derivation from the continuous time setting. The time horizon $[0,T]$ is discretized into $N$ subintervals of equal length $\Delta=T/N$ from $t_0=0$ to $t_N=T$. At maturity time $T$, the payoff of a given option contract is known to be some function $y_N:\mathbb{R}\rightarrow \mathbb{R}$. In line with the literature, we consider the pricing of a Call spread option which corresponds to the choice 
\[
y_N(x) = \max\left(0, s_0 e^{(\mu-\frac12 \sigma_r^2)T + \sigma_r x} - K_1\right)-\max\left(0, s_0 e^{(\mu-\frac12 \sigma_r^2)T + \sigma_r x} - K_2\right).
\] 
Here, $s_0>0$ is the initial stock price, $\mu\in \mathbb{R}$ is the drift under the pricing measure, $K_1,K_2>0$ is a pair of strike prices, and $\sigma_r$ is the so-called reference volatility, a choice parameter in the discretization. Notice that $x$ takes the place of the Brownian motion driving the stock price and not that of the stock price itself. Then, for $i=1,\ldots,N$, price functions $y_{i}$ and $y_{i-1}$ at times $t_i$ and $t_{i-1}$
are related through the recursion
\begin{equation}\label{eq:recursion}
y_{i-1}(x)=G(z_{i-1,0}(x),z_{i-1,1}(x),z_{i-1,2}(x)) 
\end{equation}
where, for $\iota=0,1,2$, 
\begin{equation}\label{eq:zij}
z_{i-1,\iota}(x)=E[\mathcal{H}_{\iota,\Delta}(\xi)  y_i(X_1+ \sqrt{\Delta}\xi)|X_1=x]
\end{equation}
and where $G:\mathbb{R}^3\rightarrow \mathbb{R}$ is given by 
\begin{align*}
G(z_0,z_1,z_2)
=z_0+ \frac{\Delta}{2} (z_2-\sigma_r z_1) \left(  \frac{\sigma_h^2}{\sigma_r^2}{\bf 1}_{\{z_2>\sigma_r z_1\}}
+\frac{\sigma_l^2}{\sigma_r^2}{\bf 1}_{\{z_2\leq \sigma_r z_1\}} -1 \right). 
\end{align*}
The quantity of interest is the option price at the initial time and initial value, $y_0(0)$. In the function $G$, the terms in the round brackets can be interpreted as switching from the reference volatility $\sigma_r$ to either $\sigma_h$ or $\sigma_l$ for the time interval from $t_{i-1}$ to $t_i$. Which of these two alternatives is chosen depends on the terms $z_1$ and $z_2$ which correspond to the first two derivatives. Thus, the pricing recursion depends in a non-linear way on the second derivative, underlining the fact that it is a discretization of a fully non-linear partial differential equation.

Evidently, the difficult part in solving the recursion \eqref{eq:recursion} numerically is the computation of the conditional expectations in \eqref{eq:zij}. This is exactly the problem RAWBFST is designed for. We thus propose the following algorithm for constructing a sequence $(\hat{y}_i)_i$, $i=0,\ldots,N$ of real-valued functions that approximate $(y_i)_i$. 

\begin{alg}\label{alg:3} {\hspace{1pt} \vspace{-10pt}}\\
\begin{itemize}
\item Initialization: $\hat{y}_N\equiv y_N$.
\item For $i=N,\ldots,1$:
	\begin{itemize}
		\item For $\iota=0,1,2$: call Algorithm \ref{alg:1} with input $y \equiv \hat{y}_{i}$ and output $\hat{z}_{i-1,\iota}$.
		\item  Define $\hat{y}_{i-1}$ via $\hat{y}_{i-1}(x)=G(\hat{z}_{i-1,0}(x),\hat{z}_{i-1,1}(x),\hat{z}_{i-1,2}(x))$.
	\end{itemize}
\item Return $\hat{y}_{0}(0)$.
\end{itemize}
\end{alg}

\begin{rem}
(i) This algorithm is very similar to the usual LSMC algorithms for BSDEs. We simply replace `regression now' (as, e.g., in \cite{LGW}) or `regression later' (as in \cite{BS}) by RAWBFST, Algorithm \ref{alg:1}. 

(ii) In line with the formal statement of Algorithm \ref{alg:1} above, we call it three times in each time step of  Algorithm  \ref{alg:3}, once for every relevant value of $\iota$.  In our practical implementation, we exploit that only the evaluation in Step 5 of the algorithm depends on $\iota$.  Thus, the first four steps of computing the regression coefficients need to be performed only once.

(iii) As discussed, e.g., in \cite{GLTV}, memory usage and scope for parallelization are potential bottlenecks in LSMC algorithms. In both regards, Algorithm \ref{alg:3} has excellent properties. Monte Carlo samples are generated only within the calls of Algorithm \ref{alg:1}. No $N$-step sample trajectories are produced or stored over time. Moreover, the regressions within each cube can be computed independently. Thus, in principle, one can implement Algorithm \ref{alg:1} in such a way that (only) the full set of $(1+Q)|I|$ regression coefficients from the previous step is communicated to one processor for each of the $|I|$ cubes. This processor simulates $L$ samples and computes and returns the $1+Q$ coefficients that define the local polynomials on this cube. This results in very modest memory requirements for typical choices of $Q$, $|I|$ and $L$.
\end{rem}

In order to set the parameters for RAWBFST, we need to make some `guesses' about the time discretization error of the approximation scheme \eqref{eq:recursion}--\eqref{eq:zij} to the Black-Scholes-Barenblatt equation and about the error propagation, which results from 
nesting the RAWBFST approximation of the true conditional expectations backwards in time. For the time-discretization error it is known from \cite{GL}, 
that the probabilistic scheme converges at the order $\Delta$ in the case of a quasi-linear parabolic PDE,
if the coefficient functions are sufficiently smooth and the forward SDE can be sampled without discretization error. Our numerical results below support this convergence behavior of the time discretization error in the UVM test case, although it is not backed by the theoretical error analysis for the non-linear case in \cite{FTW}. The error propagation 
for the approximation of the conditional expectations is of the order $\Delta^{-1/2}$ for `regression now' in the quasi-linear case, see \cite{LGW}. For the parameter choice of the algorithm, we here assume that the same is true for RAWBFST in the non-linear case.
Hence, we may hope for an overall convergence of the order $\Delta$, if the conditional expectations in \eqref{eq:zij} are approximated to the order $\Delta^{3/2}$ for $\iota=0$ and to the order $\Delta$ for $\iota=1,2$. In line with Theorem \ref{thm:main} and applying 
local polynomials of degree up to $Q=4$, 
we, thus, set  $\gamma_{\textnormal{cube}}=0.4$, $\gamma_{1,\textnormal{trunc}}=3$, $\gamma_{2,\textnormal{trunc}}=6$. We choose a slightly finer space discretization than before, setting $c_{\textnormal{cube}}=2$ 
while keeping $c_{1,\textnormal{trunc}}=c_{2,\textnormal{trunc}}=5$. The above parameter choice implies $c^*_{\textnormal{paths}}(Q,1)=67.33$. We then choose $c_{1,\textnormal{paths}}= 1.1\, 
c^*_{\textnormal{paths}}(Q,1)=74.07$, $c_{2,\textnormal{paths}}=1$ and, thus, $L=\lceil 3 \cdot c_{1,\textnormal{paths}} \log(\Delta^{-1})\rceil$.  As before,  we let $\tau = 0.0233$.

The only parameter we choose adaptively at each step $i$ is the parameter $C_{2,f}$. Since the underlying state process $X$ is a Brownian motion with $x_0=0$, the mechanical choice would be to set $C_{2,f}=t_{i-1}=(i-1) \Delta$. This corresponds to the variance of the Brownian motion at time $t_{i-1}$  which takes the role of $X_1$ in Algorithm \ref{alg:1} at step $i$. However, this would lead to degeneration at time $i=0$. Intuitively, we need to approximate the functions $y_i$ in a small interval around $0$ if we wish to approximate derivatives in $0$ well. We thus choose $C_{2,f}=\sigma_0^2 +t_{i-1}$ in our implementation, $\sigma_0^2=0.1 T$. This corresponds to replacing our standard Brownian motion by one that was started in $0$ at time $-\sigma_0^2$. For the parameters of the spread option and the UVM, we follow \cite{GHL} and the subsequent literature, choosing $s_0=100$, $\mu=0$, $\sigma_l=0.1$, $\sigma_h=0.2$, $\sigma_r=0.15$, $T=1$, $K_1=90$ and $K_2=110$. In this setting, the continuous-time limit $\Delta \downarrow 0$ of ${y}_{0}(0)$ is given by  11.20456 as shown in \cite{Vanden} which provides closed-form pricing formulas for this type of product in the UVM. 

The black line in Figure \ref{fig:UVM1} shows the mean squared error $\mathcal{E}(\Delta):=\hat{E}\left[(\hat{y}_{0}(0)-11.20456 )^2\right]^\frac12$ in a log-log plot against the number of time steps $\Delta=2^{-n}$, $n=4,\ldots 12$. Here, the empirical mean $\hat{E}$ denotes an average over 100 independent realizations of $\hat{y}_{0}(0)$. We see that the algorithm converges at a rate which is similar to the expected rate of 1 in the stepsize $\Delta$, which is depicted in gray. Figure \ref{fig:UVM2} gives an analogous plot of $\mathcal{E}(\Delta)$ against the run time of the algorithm. In our implementation, we have $\gamma_{\textnormal{cube}}=0.4$. As the number of time steps behaves like $\Delta^{-1}$, we expect run time to behave like $\Delta^{-1.4}$. This suggests a convergence rate of $5/7$ for $\mathcal{E}(\Delta)$ against run time. The gray line with slope $-5/7$ demonstrates that the empirical convergence rate is very much in line with this reasoning. 

Table \ref{tab2} reports further summary statistics such as the mean and standard deviation of $\hat{y}_0(0)$. Together with the figures, the table confirms that
our approximation converges stably towards its limit at a rate that is faster than standard Monte Carlo for the approximation of a single expectation, and is in line with our heuristics. This is in marked contrast to earlier implementations of this example in \cite{GHL, AA, BGS2} which show that a stable approximation at $\Delta^{-1}=4096$ should not be taken for granted for regression-based Monte Carlo algorithms.  

\begin{figure}
\centering
\includegraphics[height=200pt]{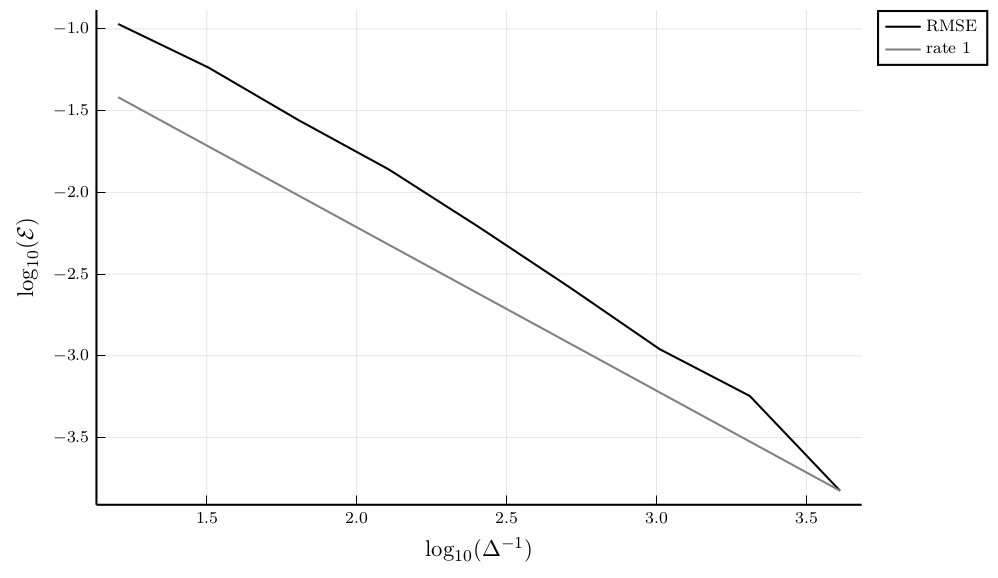}
\caption{Approximation errors against $\Delta^{-1}$ in a $\log_{10}$-$\log_{10}$-plot.} 
\label{fig:UVM1}
\end{figure}

\begin{figure}
\centering
\includegraphics[height=200pt]{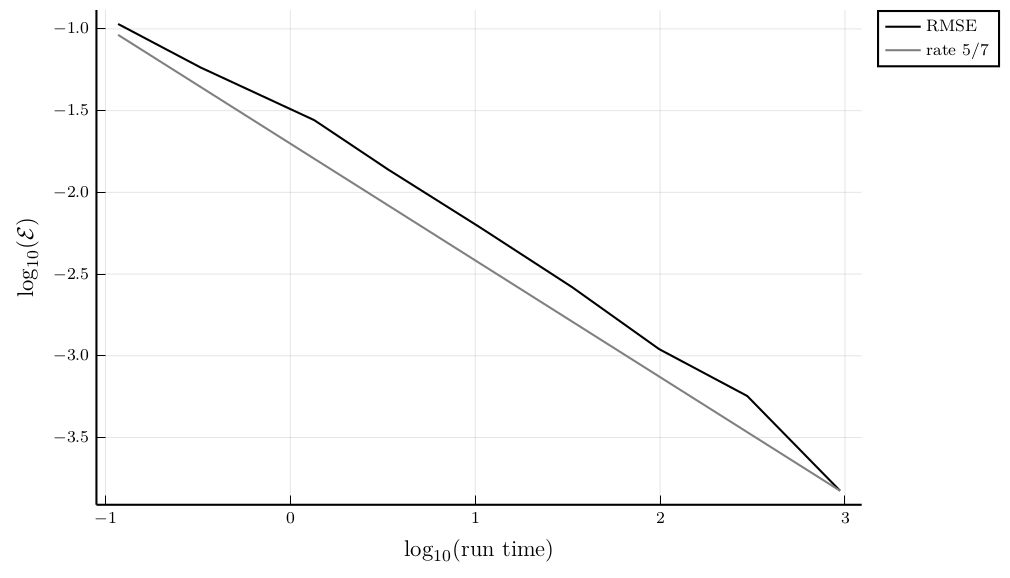}
\caption{Approximation errors against run time in a $\log_{10}$-$\log_{10}$-plot. Run times are for a Julia 1.4.2 implementation on a Windows desktop PC with an Intel Core i7-6700 CPU with 3.4GHz.} 
\label{fig:UVM2}
\end{figure}

\begin{table}
\centering
\begin{tabular}{cccc}
\hline
\vspace{-8pt} &&&\\
$\Delta^{-1}$ & mean & standard deviation & run time in $s$\\
\hline
\vspace{-8pt} &&&\\
16 & 11.0979&$1.10\cdot 10^{-2}$& 0.12\\
32 & 11.1466&$3.63\cdot 10^{-3}$& 0.33\\
64& 11.1770 &$2.14\cdot 10^{-3}$& 1.35\\
128& 11.1908&$9.93\cdot 10^{-4}$& 3.38\\
256& 11.1985&$5.81\cdot 10^{-4}$& 10.58\\
512& 11.2019&$2.76\cdot 10^{-4}$& 33.40\\
1024& 11.2035&$1.86\cdot 10^{-4}$& 98.78\\
2048& 11.2040&$8.56\cdot 10^{-5}$& 296.07\\
4096& 11.2044&$4.49\cdot 10^{-5}$& 940.70\\
\hline
\end{tabular}
\caption{Mean and standard deviation of $\hat{y}_0(0)$ across 100 runs of the algorithm.  Run times are for a single run of the algorithm. }\label{tab2}
\end{table}

\subsection{A multidimensional BSDE}\label{sec:appBSDE}
As a third illustration, we apply Algorithm \ref{alg:1} to the solution of a $D$-dimensional first-order BSDE due to \cite{GLTV}. We refer to Section 5 of their paper for details of the setting and note that, after discretization, the problem corresponds to recursively computing a sequence of functions $y_{N-1},\ldots,y_0:\mathbb{R}^D \rightarrow \mathbb{R}$ from a terminal condition $y_N$ similarly to the previous section. The time horizon $[0,T]$ is discretized into $N$ intervals of equal length $\Delta=T/N$. The terminal condition $y_N:\mathbb{R}^D \rightarrow \mathbb{R}$ is $y_N(x)=(1+\exp(-T-\sum_{d=1}^D x_d))^{-1}$. The functions $y_i$ and $y_{i-1}$ are related through
\[
y_{i-1}(x)=G(z_{i-1,0}(x),\ldots,z_{i-1,D}(x)) 
\]
where, for $d=0,\ldots,D$ and $x\in \mathbb{R}^d$ 
\begin{equation} \label{eq:BSDE_ce}
z_{i-1,d}(x)=E[\mathcal{H}_{\iota_d,\Delta}(\xi)  y_i(X_1+ \sqrt{\Delta}\xi)|X_1=x]
\end{equation}
and $\xi$ is $D$-dimensional standard normal. The multi-indices $\iota_d$ for $d\geq 1$ are the unit vectors in $\mathbb{N}_0^D$ while $\iota_0=0\in \mathbb{N}_0^D$. The function $G:\mathbb{R}^{1+D}\rightarrow \mathbb{R}$ is given by 
\begin{align*}
G(z_0,\ldots,z_D)=z_0 +\Delta  \left(z_0 -\frac{1}{D}-\frac{1}{2} \right)\sum_{d=1}^D z_d.
\end{align*}
This backward recursion thus corresponds to a time discretization of a first-order BSDE driven by a $D$-dimensional Brownian motion $X$. 
Relying on the well-known connection between first-order BSDEs and semi-linear parabolic partial differential equations, one can
show that the sequence of functions $y_i$ can be understood as a time discretization of the solution $u:[0,T] \times \mathbb{R}^D \rightarrow \R$ to the partial differential equation
\[
0 = \frac{\partial u}{\partial t} + \sum_{d=1}^D \frac{\partial^2 u}{\partial x_d^2}  +   \left(u -\frac{1}{D}-\frac{1}{2}  \right) \sum_{d=1}^D  \frac{\partial u}{\partial x_d}
\]
with terminal condition $u(T,\cdot)=y_N(\cdot)$.  Our goal is to compute $u(0,0)$, the continuous-time limit of $y_0(0)$, which is equal to $1/2$ as stated in \cite{GLTV}. 

To compute approximate solutions $\hat{y}_{N-1},\ldots, \hat{y}_{0}$, we adapt Algorithm \ref{alg:3} in the obvious way. 
Throughout, we fix $T=1$ and $D=5$. Due to the higher dimensionality, we adjust some parameters of the implementation of
the previous section to achieve an efficient allocation of computational effort across time and space. 
Calibrating the algorithm to an error of the order $\Delta^{1/2}$ and taking again the error propagation in the dynamic programming equation heuristically into account,
we wish to approximate the conditional expectations in \eqref{eq:BSDE_ce} to the order $\Delta$ for $d=0$ and to the order
$\Delta^{1/2}$ for $d=1,\ldots,D$. This corresponds to the choices $\rho=2$ and $\rho=1$ in Theorem \ref{thm:main}, respectively.
We thus choose $Q=3$, $\gamma_{\textnormal{cube}}=1/4$, $\gamma_{1,\textnormal{trunc}}=2$, and $\gamma_{2,\textnormal{trunc}}=3$
In particular, the number of cubes 
grows as $\Delta^{-5/4}$. We 
keep $c_{\textnormal{cube}}=2$ and $c_{2,\textnormal{trunc}}=5$ from the previous section but slightly decrease $r_1$ by 
increasing $c_{1,\textnormal{trunc}}$ to 20. For the constant $c^*_{\textnormal{paths}}(Q,D)$ we 
find $c^*_{\textnormal{paths}}(Q,D)=5890/3$. We keep the choices of $c_{1,\textnormal{paths}}= 1.1\,c^*_{\textnormal{paths}}(Q,D)$ and
thus $\tau = 0.0233$ but  decrease the number of trajectories by setting $c_{2,\textnormal{paths}}=0.5$. The choice of $L$ then follows 
Theorem \ref{thm:main}. For the adaptive truncation, we set $C_{2,f}=\sigma_0^2 +t_{i-1}$ with $\sigma_0^2=0.1$ as before. 
Finally, compared to the construction in Algorithm \ref{alg:1}, we shift the grid in space such that $0\in \mathbb{R}^D$ lies in the 
center of a cube. 

For comparison with the `regression now' algorithm of \cite{GLTV} we note that tailoring the latter algorithm to the same accuracy 
requires a grid in which the number of cubes also increases like $\Delta^{-5/4}$. However, the simulation cost per cube 
is of the order $\Delta^{-3}$, while in the RAWBFST algorithm the simulation cost per cube grows logarithmically in $\Delta^{-1}$.

We define the approximation error $\mathcal{E}(\Delta)$ like in the previous section but reduce the number of independent copies of $\hat{y}_0(0)$ for each value of $\Delta^{-1}$ to 20.

\begin{table}
\centering
\begin{tabular}{cccc}
\hline
\vspace{-8pt} &&&\\
$\Delta^{-1}$ & mean & standard deviation & run time in $s$\\
\hline
\vspace{-8pt} &&&\\
 10&  0.486427&  $5.01 \cdot 10^{-4}$ &    161\\
 20&  0.493735&  $2.52 \cdot 10^{-4}$ &  2031\\
 30&  0.497602&  $1.34 \cdot 10^{-4}$ &  7941\\
 40&  0.499836&  $8.33 \cdot 10^{-5}$   & 20820\\
 50&  0.501483&  $8.01 \cdot 10^{-5}$  & 43794\\
 60&  0.501333&  $8.01 \cdot 10^{-5}$   & 81339\\
 70&  0.501016&  $5.77 \cdot 10^{-5}$   &140032\\
\hline
\end{tabular}
\caption{Mean and standard deviation of $\hat{y}_0(0)$ across 20 runs of the algorithm.  Run times are for a single run of the algorithm. }\label{tab3}
\end{table}

\begin{figure}
\centering
\includegraphics[height=200pt]{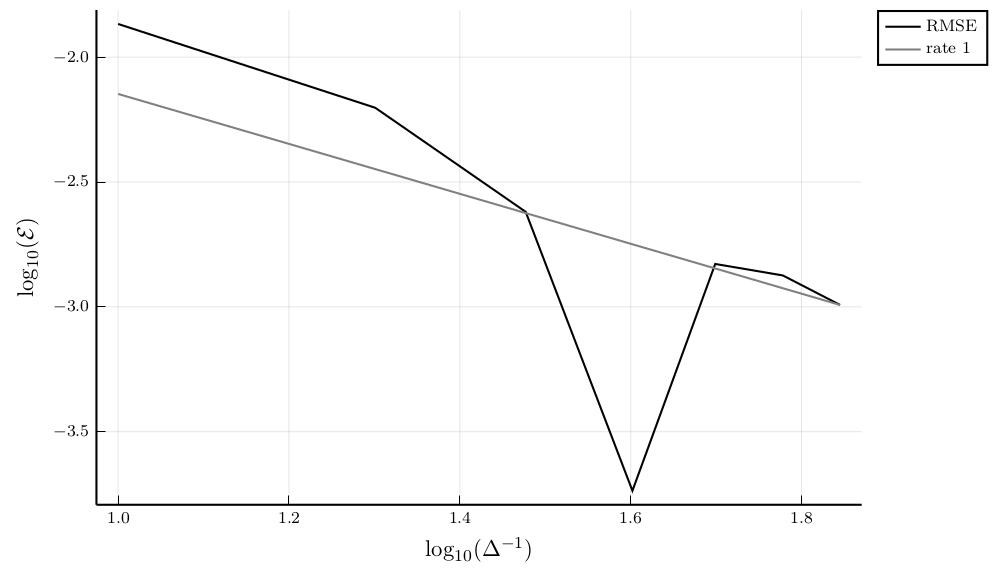}
\caption{Approximation errors against $\Delta^{-1}$ in a $\log_{10}$-$\log_{10}$-plot.} 
\label{fig:Gobet1}
\end{figure}

\begin{figure}
\centering
\includegraphics[height=200pt]{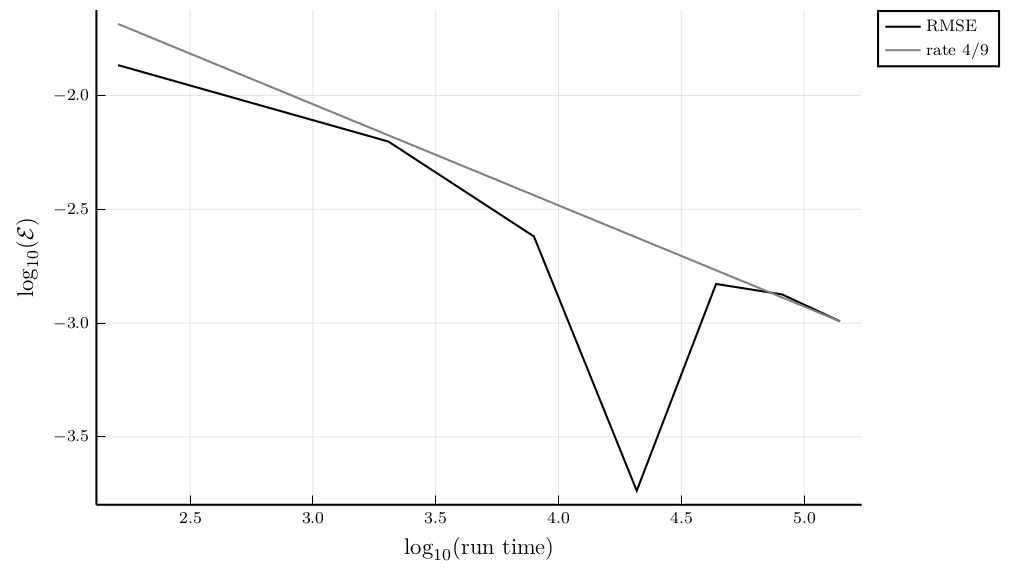}
\caption{Approximation errors against run time in a $\log_{10}$-$\log_{10}$-plot. Run times are for a Julia 1.4.2 implementation on a Windows desktop PC with an Intel Core i7-6700 CPU with 3.4GHz.} 
\label{fig:Gobet2}
\end{figure}

As seen in in Table \ref{tab3}, the bias, i.e. the distance to 0.5, switches its sign around $\Delta^{-1}=40$ which explains the steep local drop in the RMSE we see in Figure \ref{fig:Gobet1} where we plot the RMSE against the number of time steps. If we ignore this single point, the decay of the RMSE against  $\Delta^{-1}$ in the log-log-plot seems to be well-approximated by a line with slope -1 (or even steeper). This convergence rate is also largely supported by the behavior of standard deviations against $\Delta^{-1}$. For instance, the ratio between the standard deviations for $\Delta^{-1}=20$ and $\Delta^{-1}=60$ is close to the expected factor 3. 

Up to logarithmic factors,  we expect the overall run time to scale like $N^{1+D\gamma_{\textnormal{cube}}} = \Delta^{-9/4}$. By this reasoning, if we take the empirical convergence of order $\Delta$ in Figure \ref{fig:Gobet1} as a given, we expect a convergence rate $4/9$ against run time in Figure \ref{fig:Gobet2} which plots errors against run time. As seen from the gray line, this is in good agreement with the observed results.

\section{Least-squares interpolation with brute-force SVD truncation}\label{sec:interpolation}

In this section, we analyze the main building block of Algorithm \ref{alg:1}, namely the least-squares regression with brute-force SVD truncation, when there is no noise in the dependent variable. We also show that (up to log-factors) optimal
rates for the interpolation problem with fixed random design can be achieved under suitable assumptions.

\subsection{Algorithm and convergence analysis}

Suppose $X$ is an $\R^D$-valued random variable with law $\mu$ and $Y=y(X)$ for some measurable function $y:\R^D\rightarrow \R$. Given i.i.d copies $\mathcal{D}=(X^{(l)},Y^{(l)})_{l=1,\ldots,L}$, our aim is to estimate the function $y$. 
This problem can be interpreted as an interpolation problem with fixed random design, see \cite{KK, BDKKW}, and has been studied 
in the context of least-squares regression in \cite{Bal,CDL,CM} and the references therein.

We propose a modified least-squares approach, which utilizes a brute-force truncation of the singular value decomposition. To this end, let $\eta(x)=(\eta_1(x),\ldots,\eta_K(x))^\top$ denote a vector of basis functions. 
\begin{alg}\label{alg:2}
Fix a threshold $\tau>0$. 
\begin{itemize}
 \item Build the empirical regression matrix
$$
A=(\eta_k(X^{(l)}))_{l=1,\ldots,L;\;k=1,\ldots K}
$$
and perform a singular value decomposition of $A^\top$:
$$
A^\top=\U\D\V,\quad \U\in O(K),\; \V\in O(L),
$$
where $\D$ is the $K\times L$-matrix which has the singular values $s_1\geq s_2\geq \cdots\geq s_K\geq 0$ of $A$ on the diagonal and has zero entries otherwise.
\item If $s_K^2\geq L \tau$, let
$$
\alpha_L=\U\D^\dagger \V \;(Y^{(1)},\ldots, Y^{(L)})^\top
$$
where the pseudoinverse $\D^\dagger$ of $\D^\top$ is the $K\times L$-matrix which has $s_1^{-1},\ldots,s_K^{-1}$  on the diagonal and has zero entries otherwise.

 If $s_K^2<L \tau$, let
$$
\alpha_L=0\in\mathbb{R}^K.
$$
\item Return 
$$
\hat y_L(x):=\alpha_L^\top \eta(x)=\sum_{k=1}^K \alpha_{L,k} \eta_k(x)
$$
as an approximation of $y$.
\end{itemize}
\end{alg}

\begin{rem}\label{rem:svd}
(i) Truncated singular value decomposition is a popular regularization method for ill-conditioned linear regression problems, see e.g. \cite{Ha}. One approximates 
the solution to the regression problem by using the largest $t$ singular values only, i.e. one chooses the coefficient vector
$$
\hat \alpha_t=\U\D_t^\dagger \V \;(Y^{(1)},\ldots, Y^{(L)})^\top,
$$
where $\D_t^\dagger$ is the  $K\times L$-matrix, which has $s_1^{-1},\ldots,s_t^{-1},0,\ldots,0$ on the diagonal and has zero entries otherwise, and $t$ can be interpreted as a regularization parameter.  In contrast, in our brute-force SVD truncation, we either keep all singular values (if the smallest singular value is sufficiently large),
or we completely discard all the singular values (otherwise). 

(ii) In practical implementations with $L \gg K$, we recommend replacing the SVD above by the thin SVD (\cite{GVL}, p.72). 
To this end, consider the  SVD $A^\top=\U\D\V$ from above. Denote by $\bar{\D}$ and $\bar{\D}^\dagger$ the $K\times K$ matrices containing the first $K$ 
columns of $\D$ and $\D^\dagger$ and by $\bar{\V}$ the $K \times L$ matrix containing the first $K$ rows of $\V$. Due to the identities $\D\V=\bar{\D}\bar{\V}$ and $\D^\dagger\V=\bar{\D}^\dagger\bar{\V}$, it 
suffices to compute $\bar{\V}$ and $\bar{\D}$ when implementing Algorithm \ref{alg:2}, avoiding the large $L\times L$ matrix $\V$. 
\end{rem}

For the error analysis, 
we write
$$
R=(E[\eta_k(X)\eta_\kappa(X)])_{k,\kappa=1,\ldots,K}.
$$ 
We assume that $R$ has full rank and that we have access to bounds on the extremal eigenvalues of $R$:
$$
0<\lambda_*\leq \lambda_{\min}(R) \leq \lambda_{\max}(R)\leq \lambda^*<\infty.
$$

\begin{thm}\label{thm:noiseless}
 Suppose that the basis functions $\eta_k$ are bounded. Let
$$
m:=\sup_{x\in \supp(X)} \sum_{k=1}^K |\eta_k(x)|^2
$$
and $\tau=(1-\epsilon)\lambda_*$ for some $\epsilon\in (0,1)$.
Then,
\begin{eqnarray*}
 && E\left[\int_{\R^D} |y(x)-\hat y_L(x)|^2 \mu(dx) \right]\\ &\leq& \left(1+ \frac{\lambda^*}{\lambda_*(1-\epsilon)}\right) \inf_{\alpha\in \R^K} \int_{\R^D} |y(x)-\alpha^\top \eta(x)|^2 \mu(dx)  \\
&&+  2K\exp\left\{-\frac{3 \epsilon^2 L}{6 m\lambda^*/\lambda_*^2+ 2\epsilon(m/\lambda_*+\lambda^*/\lambda_*)} \right\} \int_{\R^D} |y(x)|^2 \mu(dx) ,
\end{eqnarray*}
where $\hat y_L$ is constructed by Algorithm \ref{alg:2}. 
\end{thm}
\begin{rem}
 (i) A similar result for orthonormal basis functions has been derived in \cite{CDL}, which is not applicable in our context, since the 
 RAWBFST basis functions are not orthonormal with respect to the law of $X_2$. In contrast to the arguments in \cite{CDL}, our proof 
 relies on the matrix Bernstein inequality.
 Compared to standard least-squares regression estimates \cite{Gal}, an important feature of the algorithm is that for a fixed approximation architecture (i.e., a fixed function basis), the statistical error converges exponentially in the number $L$ of samples. \\[0.1cm]
(ii) If one scales the basis functions by a multiplicative constant $\gamma\neq 0$, then $m$, $\lambda_{\max}(R)$, and $\lambda_{\min}(R)$ are scaled by the factor $\gamma^2$. Hence, the error analysis in the above theorem 
is invariant against scaling of the basis functions. \\[0.1cm]
(iii) The error analysis in the above theorem is not distribution-free, but it depends on the distribution of $X$ only through bounds on the extremal eigenvalues of the matrix $R$. It is shown in \cite{KK} 
that the optimal rates for the interpolation problem with a random design, when the samples are drawn from a uniform distribution on the unit cube, are not valid for general distributions of $X$ on the unit cube. Hence, some 
dependence of the error bounds in Theorem \ref{thm:noiseless} on the distribution of $X$ cannot be avoided.
\end{rem}

The proof is prepared by several lemmas. These lemmas do not require the basis functions to be bounded, but they are merely assumed to be square-integrable with respect to the law of $X$.
The first lemma reduces the problem to a problem of estimating the SVD truncation probability.

\begin{lem}\label{lem:error_decomposition}
In the setting of Theorem \ref{thm:noiseless},
 \begin{eqnarray*}
 && E\left[\int_{\R^D} |y(x)-\hat y_L(x)|^2 \mu(dx) \right]\\ &\leq& \left(1+ \frac{\lambda^*}{\lambda_*(1-\epsilon)}\right) \inf_{\alpha\in \R^K} \int_{\R^D} |y(x)-\alpha^\top \eta(x)|^2 \mu(dx)  \\
&&+  P(\{s_K^2<L (1-\epsilon) \lambda_{*} \}) \int_{\R^D} |y(x)|^2 \mu(dx).
\end{eqnarray*}
\end{lem}
\begin{proof}
 We decompose
\begin{eqnarray*}
 && E\left[\int_{\R^D} |y(x)-\hat y_L(x)|^2 \mu(dx) \right]= E\left[\int_{\R^D} |y(x)-\hat y_L(x)|^2 \mu(dx) {\bf 1}_{\{ s_K^2\geq L (1-\epsilon) \lambda_{*}\}} \right] \\
& + & E\left[\int_{\R^D} |y(x)-\hat y_L(x)|^2 \mu(dx) {\bf 1}_{\{ s_K^2< L (1-\epsilon) \lambda_{*}\}} \right]=:(I)+(II).
\end{eqnarray*}
As $\hat y_L=0$ on $\{ s_K^2< L (1-\epsilon) \lambda_{*}\}$, we obtain
$$
(II)=  P(\{s_K^2<L (1-\epsilon) \lambda_{*} \}) \int_{\R^D} |y(x)|^2 \mu(dx).
$$
We now treat term (I). We are going to show that 
\begin{eqnarray}\label{eq:hilf013}
 && E\left[\int_{\R^D} |y(x)-\hat y_L(x)|^2 \mu(dx) {\bf 1}_{\{ s_K^2\geq L (1-\epsilon) \lambda_{*}\}} \right] \nonumber \\
 &\leq&  \left(1+ \frac{\lambda^*}{\lambda_*(1-\epsilon)}\right) \inf_{\alpha\in \R^K} \int_{\R^D} |y(x)-\alpha^\top \eta(x)|^2 \mu(dx). 
\end{eqnarray}
To this end we denote by $\alpha_*^\top\eta(X)$ the orthogonal  projection of $Y$ on $\Span(\eta_k(X);\,k=1,\ldots,K)$. The full rank condition on $R$ ensures that 
$\alpha_*$ is uniquely determined. Then, by orthogonality,
\begin{eqnarray*}
 (I)&=&  E\left[\int_{\R^D} |(y(x)-\alpha_*^\top\eta(x))+(\alpha_*^\top\eta(x)-\hat y_L(x))|^2 \mu(dx) {\bf 1}_{\{ s_K^2\geq L (1-\epsilon) \lambda_{*}\}} \right]
\\ &\leq& E\left[\int_{\R^D} |(\alpha_L-\alpha_*)^\top\eta(x)|^2 \mu(dx) {\bf 1}_{\{ s_K^2\geq L (1-\epsilon) \lambda_{*}\}} \right]\\ &&+ \inf_{\alpha\in \R^K} \int_{\R^D} |y(x)-\alpha^\top \eta(x)|^2 \mu(dx)
\end{eqnarray*}
 It remains to show that
\begin{eqnarray}\label{eq:hilf001}
 &&  E\left[\int_{\R^D} |(\alpha_L-\alpha_*)^\top\eta(x)|^2 \mu(dx) {\bf 1}_{\{ s_K^2\geq L (1-\epsilon) \lambda_{*}\}} \right] \nonumber \\
&\leq &  \frac{\lambda^*}{\lambda_*(1-\epsilon)} \inf_{\alpha\in \R^K} \int_{\R^D} |y(x)-\alpha^\top \eta(x)|^2 \mu(dx)
\end{eqnarray}
On the set $\{ s_K^2\geq L (1-\epsilon) \lambda_{*}\}$, the empirical regression matrix $A$ has rank $K$. Hence, $\alpha_*$ is the unique solution $z$ to the linear system
$$
Az=(\alpha_*^\top\eta(X^{(l)}))_{l=1,\ldots,L}.
$$ 
As $\U\D^\dagger \V$ is the pseudoinverse of $A$, we obtain,
$$
\alpha_L-\alpha_*= \U\D^\dagger \V\,\xi_L,\quad \xi_L:=(Y^{(l)}-\alpha_*^\top\eta(X^{(l)}))_{l=1,\ldots,L}.
$$ 
Thus, on the set $\{ s_K^2\geq L (1-\epsilon) \lambda_{*}\}$,
\begin{eqnarray}\label{eq:hilf014}
 && \int_{\R^D} |(\alpha_L-\alpha_*)^\top\eta(x)|^2 \mu(dx)= (\alpha_L-\alpha_*)^\top \left( \int_{\R^D} \eta(x)\eta(x)^\top\mu(dx)\right) (\alpha_L-\alpha_*) \nonumber \\
&=&  \xi_L^\top \V^\top (\D^\dagger)^\top  \U^\top R \U\D^\dagger \V\,\xi_L \leq \lambda_{\max}(R) \, \xi_L^\top \V^\top (\D^\dagger)^\top \D^\dagger \V\,\xi_L  \nonumber\\ &\leq& 
\lambda_{\max}(R) \lambda_{\max}((\D^\dagger)^\top \D^\dagger) \, \xi_L^\top\xi_L \leq \frac{\lambda^*}{s_K^{2}} \sum_{l=1}^L  (Y^{(l)}-\alpha_*^\top\eta(X^{(l)}))^2,
\end{eqnarray}
because $(\D^\dagger)^\top \D^\dagger=diag(s_1^{-2},\ldots,s_K^{-2},0,\ldots,0)$. Hence,
\begin{eqnarray*}
  && E\left[\int_{\R^D} |(\alpha_L-\alpha_*)^\top\eta(x)|^2 \mu(dx) {\bf 1}_{\{ s_K^2\geq L (1-\epsilon) \lambda_{*}\}} \right]
\\ &\leq &  \frac{\lambda^*}{\lambda_*(1-\epsilon)} E\left[\frac{1}{L} \sum_{l=1}^L  (Y^{(l)}-\alpha_*^\top\eta(X^{(l)}))^2\right]\\ &=&\frac{\lambda^*}{\lambda_*(1-\epsilon)} \inf_{\alpha\in \R^K} \int_{\R^D} |y(x)-\alpha^\top \eta(x)|^2 \mu(dx).
\end{eqnarray*}
Thus, \eqref{eq:hilf001} holds and the proof of the lemma is complete.
\end{proof}

\begin{rem}
 One can straightforwardly modify the estimate in \eqref{eq:hilf014} in order to show that, thanks to the SVD truncation, 
 $$
 |\alpha_L|_2^2\leq  \frac{\| y\|_\infty}{\lambda_*(1-\epsilon)} .
 $$
 Hence, the SVD truncation implies a control on the Euclidean norm of the coefficient vector $\alpha_L$ resulting from the empirical regression. A related approach can be found in \cite{CO}, where the sample is rejected, if 
 the Euclidean norm of the coefficient vector exceeds some given threshold. The main advantage of our approach is that we can apply concentration inequalities for random matrices to estimate the probability that 
 the SVD truncation takes place, see the proof of Theorem \ref{thm:noiseless} below. In contrast, the probability that a sample is rejected is only discussed heuristically in \cite{CO}.
\end{rem}

\begin{lem}\label{lem:ev_vs_spectral_norm}
 $$
P(\{s_K^2<L (1-\epsilon) \lambda_{*} \}) \leq P(\{\|\frac{1}{L}(A^\top A)-R \|_2>\epsilon \lambda_* \}).
$$
\end{lem}
\begin{proof}
 Note that $s_K^2$ is the smallest eigenvalue of $A^\top A$, and, then, $s_K^2/L=\lambda_{\min}( \frac{1}{L}(A^\top A))$. Hence,
\begin{eqnarray*}
P(\{s_K^2<L (1-\epsilon) \lambda_{*} \})&\leq &  P(\{\lambda_{\min}(\frac{1}{L}(A^\top A))-\lambda_{\min}(R) <-\epsilon \lambda_{*} \}) \\ &\leq &
P(\{|\lambda_{\min}(\frac{1}{L}(A^\top A))-\lambda_{\min}(R)| >\epsilon \lambda_{*}\}. 
\end{eqnarray*}
Thus, the lemma follows from the fact that for positive semi-definite matrices $\Sigma_1,\Sigma_2$,
$$
|\lambda_{\min}(\Sigma_1)-\lambda_{\min}(\Sigma_2)|\leq \|\Sigma_1-\Sigma_2 \|_2,
$$
see e.g. \cite{HJ}, Corollary 7.3.8.
\end{proof}

The expression on the right-hand side in Lemma \ref{lem:ev_vs_spectral_norm} can be estimated by a matrix Bernstein inequality. The following version is due to 
Tropp \cite{Tropp}:
\begin{thm}[\cite{Tropp}, Theorem 1.6]\label{thm:matrix_Bernstein}
Consider a finite sequence $(Z_l)$
of independent, random matrices of size $K_1\times K_2$ . Assume that each random
matrix satisfies
$$
E[Z_l]=0,\quad  \| Z_l\|_2\leq B,\; P\textnormal{-almost surely,}
$$
for some constant $B\geq 0$.
Let
$$
\sigma^2:=\max\left\{ \left\|\sum_l E[Z_lZ_l^\top]\right\|_2,\;  \left\|\sum_l E[Z_l^\top Z_l]\right\|_2  \right\} 
$$
Then, for every $t\geq 0$,
$$
P(\{ \|\sum_l Z_l\|_2\geq t \} )\leq (K_1+K_2)\exp\left\{-\frac{t^2/2}{\sigma^2+Bt/3}\right\}.
$$
\end{thm}

\begin{proof}[Proof of Theorem \ref{thm:noiseless}]
 In view of Lemmas \ref{lem:error_decomposition} and \ref{lem:ev_vs_spectral_norm}, it suffices to show that
\begin{eqnarray}\label{eq:hilf002}
P(\{\|\frac{1}{L}(A^\top A)-R \|_2>\epsilon \lambda_* \})\leq  2K\exp\left\{-\frac{3 \epsilon^2 L}{6 m\lambda^*/\lambda_*^2+ 2\epsilon(m/\lambda_*+\lambda^*/\lambda_*)} \right\}.
\end{eqnarray}
We consider the random $K\times K$ symmetric matrices
$$
Z_l=\eta(X^{(l)})\eta(X^{(l)})^\top-R,\quad l=1,\ldots,L.
$$
Then,
$$
P(\{\|\frac{1}{L}(A^\top A)-R \|_2>\epsilon \lambda_* \})=P(\{\|\sum_{l=1}^L Z_l\|_2>L\epsilon \lambda_* \}).
$$
As obviously $E[Z_l]=0$, the matrix Bernstein inequality (Theorem \ref{thm:matrix_Bernstein}) yields
\begin{eqnarray}\label{eq:hilf003}
 P(\{\|\frac{1}{L}(A^\top A)-R \|_2>\epsilon \lambda_* \})\leq  2K\exp\left\{-\frac{3 \epsilon^2 L}{6 \sigma^2/(\lambda_*^2L)+ 2 B  \epsilon / \lambda_*} \right\},
\end{eqnarray}
where, by symmetry of $Z_l$,
\begin{eqnarray*}
 \sigma^2&=& \left\|\sum_{l=1}^L E[Z_l^2]\right\|_2,\quad B= \esssup_{\omega\in \Omega} \| Z_l(\omega)\|_2.
\end{eqnarray*}
Recall that for a positive semi-definite matrix $\Sigma$,
$$
\|\Sigma\|_2=\max_{u\in \R^K;\; |u|_2=1} u^\top \Sigma u =\lambda_{\max}(\Sigma).
$$
Hence, by the triangle inequality, 
\begin{eqnarray*}
 \| Z_l\|_2&\leq& \max_{u\in \R^K;\; |u|_2=1} (u^\top\eta(X^{(l)}) \eta(X^{(l)})^\top u) + \lambda_{\max}(R) =| \eta(X^{(l)})|_2^2+ \lambda_{\max}(R) \\ &\leq& m+\lambda^*.
\end{eqnarray*}
Thus,
$$
B\leq (m+\lambda^*).
$$
Now, note that the matrix $Z_l^2=Z_lZ_l^\top$ is ($\omega$-wise) positive semidefinite, and then so is the matrix $E[Z_l^2]$. As
$$
E[Z_l^2]=E[\eta(X^{(l)})\eta(X^{(l)})^\top\eta(X^{(l)})\eta(X^{(l)})^\top]-RR\leq E[\eta(X^{(l)})\eta(X^{(l)})^\top\eta(X^{(l)})\eta(X^{(l)})^\top],
$$
we obtain
\begin{eqnarray*}
 \sigma^2&\leq & L \lambda_{\max}(E[\eta(X^{(l)})\eta(X^{(l)})^\top\eta(X^{(l)})\eta(X^{(l)})^\top]) \leq Lm \lambda_{\max}(R)\leq Lm\lambda^*.
\end{eqnarray*}
Combining the estimates for $B$ and $\sigma^2$ with \eqref{eq:hilf003} we arrive at \eqref{eq:hilf002}, and the proof is complete.
\end{proof}

\subsection{Piecewise polynomial estimates}

We now consider the case of piecewise polynomial estimates, when the random variable $X$ is supported on the unit cube $[0,1]^D$. More precisely, we assume 
that $X$ has a density with respect to the Lebesgue measure in $\R^D$ of the form
$$
f(x){\bf 1}_{[0,1]^D}
$$
such that $f$ is continuous and strictly positive on the unit cube. 

Our aim in this section is twofold. First, we show that our algorithm can achieve (up to a log-term) optimal rates of convergence for the interpolation problem with fixed random design. 
Second, we illustrate how to apply Theorem \ref{thm:noiseless} in a simple setting before turning to the more involved proof of Theorem \ref{thm:main}. Basically, all we need to establish as inputs 
for the theorem are suitable bounds on eigenvalues and on the supremum norm of the basis functions.  

For fixed $N\in \N$, we apply a regular cubic partition 
$$
C_{{\bf i}}=\prod_{d=1}^D \left(\frac{i_d}{N},\frac{i_d+1}{N}\right],\quad {\bf i}=(i_1,\ldots i_D)\in \{0,\ldots, N-1\}^D,
$$
of the unit cube. We still denote by $\mathcal{L}_q:\R\rightarrow \R$ the Legendre polynomial of degree $q$, which is normalized such that $\mathcal{L}_q(1)=1$, see Step 2 of Algorithm \ref{alg:1}.
Given 
a multi-index $\jj\in \N_0^D$, we consider on each cube $C_{\bf i}$ the polynomials
$$
p_{\jj, \ii}(x)= N^{D/2} \prod_{d=1}^D \sqrt{2j_d+1} \, \mathcal{L}_{j_d}\left(2(Nx_d-i_d)-1\right).
$$
We now fix the maximal degree $Q\in \N_0$, and define the basis functions to be 
$$
\eta_{{\bf i, \bf j}}=p_{{\bf i, \bf j}} {\bf 1}_{C_{{\bf i}}},\quad {\bf i} \in \{0,\ldots, N-1\}^D, \,\jj\in \N_0^D,\, \sum_d j_d\leq Q.
$$
To simplify the notation we also write $\eta_k$, $k=1,\ldots K= N^D  {{D+Q}\choose{D}}$, for any fixed ordering of these basis functions.  We define
$$
f_*:=\inf_{x\in [0,1]^D} f(x),\quad f^*:= \sup_{x\in [0,1]^D} f(x),\quad \varpi(h):= \sup_{x,z \in [0,1]^D,\, |x-z|_\infty \leq h} |f(x)-f(z)|.
$$ 
Hence,  $\varpi$ is the modulus of continuity of the density $f$ with respect to the maximum norm. The continuity assumption on $f$ ensures that $f_*>0$, $f^*<\infty$, and $\varpi(h)\rightarrow 0$, as $h$ tends to zero. 

If we run Algorithm \ref{alg:2} in the above setting with $\tau=(1-\epsilon)f_*/2$ for an arbitrary $0<\epsilon<1$, we obtain the following convergence result.
\begin{thm}\label{thm:polynomials2}
Assume $Q=0$ or $Q\in \N$ such that 
\begin{equation}\label{eq:GC}
2{{D+Q}\choose{D}} e^{2Q} \leq \frac{f_*}{\varpi(1/N)}. 
\end{equation}
 Suppose $y:[0,1]^D\rightarrow \R$ is $(Q+\gamma,C)$-smooth, i.e. $Q$-times continuously differentiable and the partial derivatives of order $Q$ are $\gamma$-H\"older-continuous for some $0<\gamma\leq 1$ with H\"older constant $C$
(w.r.t the Euclidean norm). Let
\begin{equation}\label{eq:L1}
L \geq  \frac{N^D{{D+Q}\choose{D}}e^{2Q}(36\frac{f^*}{f_*}+4 \epsilon)+6\epsilon f^*}{3\epsilon^2 f_*} \log(c_0 N^{D+2(Q+\gamma)})
\end{equation}
for some constant $c_0>0$. Then,
\begin{eqnarray*}
 && E\left[\int_{[0,1]^D} |y(x)-\hat y_L(x)|^2 \mu(dx) \right]\\ &\leq& 
\frac{1}{N^{2(Q+\gamma)}}\left( \left(1+ \frac{3f^*}{f_*(1-\epsilon)}\right)\frac{ D^{2Q+\gamma}}{(Q!)^2} C^2+ \frac{2{{D+Q}\choose{D}}}{c_0} \int_{[0,1]^D} |y(x)|^2 \mu(dx)  \right).
\end{eqnarray*}
\end{thm}
\begin{rem}\label{rem:poly}
 (i) Let the number of samples $L$ tend to infinity, and choose $N_L$ as the largest integer such that \eqref{eq:L1} is satisfied. Then, for every fixed $Q$, \eqref{eq:GC} is satisfied for sufficiently large $L$.  
Hence, the previous theorem states $\LL^2$-convergence of the order $L^{-(Q+\gamma)/D}$ up to a logarithmic factor in the number of samples. This rate matches (up to the log-factor) the $\LL^1$-minimax lower bound 
in \cite{KK} for the uniform distribution in the class of $(Q+\gamma,C)$-smooth functions and can, thus, be considered as optimal for the interpolation problem with fixed  random design. \\[0.1cm]
(ii) No attempt has been made to optimize the constants, but the focus of Theorem \ref{thm:polynomials2} is to derive the optimal rate of convergence. In particular the constant in front of the log-factor in \eqref{eq:L1}
is very conservative. As this constant determines the convergence rate of the statistical error, the approximation error due to the basis choice actually is the leading error term in the setting of Theorem \ref{thm:polynomials2}. \\[0.1cm]
\end{rem}

The proof of Theorem \ref{thm:polynomials2}  relies on the following bounds of the supremum norm of the basis functions and of the eigenvalues.

\begin{lem}
 Let $R=E[\eta(X)\eta^\top(X)]$ and $m=\sup_{x\in [0,1]^D} |\eta(x)|_2^2$. Assume $Q=0$ or that $Q\in \N$ satisfies \eqref{eq:GC}.
Then,
$$
\lambda_{\min}(R)\geq f_*/2,\quad \lambda_{\max}(R)\leq \frac{3}{2} f^*,\quad m\leq N^D e^{2Q}  {{D+Q}\choose{D}}.
$$
\end{lem}
\begin{proof}
 As the Legendre polynomials are bounded by 1 on $[-1,1]$, we obtain for every ${\bf i} \in \{0,\ldots, N-1\}^D$ and every $j\in \N_0^D$ such that $\sum_d j_d\leq Q$,
\begin{eqnarray}\label{eq:hilf005}
 \sup_{x\in C_{{\bf i}}}  |p_{\bf i, \bf j}(x)|\leq  N^{D/2} \prod_{d=1}^D \sqrt{2j_d+1}  \leq N^{D/2} e^{\sum_d j_d} \leq N^{D/2} e^Q.
\end{eqnarray}
As the cubes are disjoint and there are ${{D+Q}\choose{D}}$ basis functions per cube, the bound on $m$ immediately follows.

If $Q=0$, then $R$ is a diagonal matrix with entries $N^D \mu(C_{{\bf i}})$, because $p_0=1$. Applying the bounds $N^{-D} f_* \leq \mu(C_{\bf i})\leq N^{-D} f^*$, we obtain that 
$\lambda_{\min}(R)\geq f_*$ and $\lambda_{\max}(R)\leq f^*$.  For $Q\in \N$, $R$ has (after re-ordering of the basis functions, if necessary) block diagonal form, because the cubes are disjoint. 
Hence, it suffices to bound the eigenvalues separately on each cube $C_{\bf i}$. In order to compute the entries of the block $R_{\bf i}$ of $R$, which stems from cube $C_{\bf i}$, we fix some point $x_{\bf i}\in C_{\bf i}$.
Then,
\begin{eqnarray*}
\nonumber && \int_{C_{\bf i}}  p_{\bf i, \bf j}(x)p_{\bf i, \bf j'}(x) \mu(dx) \\&=&f(x_{\bf i})\left( \int_{C_{\bf i}}  p_{\bf i, \bf j}(x)p_{\bf i, \bf j'}(x) dx+  \int_{C_{\bf i}}  p_{\bf i, \bf j}(x)p_{\bf i, \bf j'}(x) \frac{f(x)-f(x_{\bf i})}{f(x_{\bf i})}dx\right) \nonumber
\\ &=:& f(x_{\bf i}) \left( (I)+(II) \right)
\end{eqnarray*}
Applying the orthonormality of the scaled Legendre polynomials $\sqrt{(2q+1)/2} \, \mathcal{L}_{q}$ on $[-1,1]$ with respect to the Lebesgue measure, we observe that 
\begin{eqnarray*}
 (I)&=& \int_{[-1,1]^D}  \prod_{d=1}^D \sqrt{\frac{2j_d+1}{2}} \, \mathcal{L}_{j_d}\left(x_d\right) \sqrt{\frac{2j'_d+1}{2}} \, \mathcal{L}_{j'_d}\left(x_d\right) dx
\\ &=&  \prod_{d=1}^D \int_{[-1,1]}  \sqrt{\frac{2j_d+1}{2}} \, \mathcal{L}_{j_d}\left(u\right) \sqrt{\frac{2j'_d+1}{2}} \, \mathcal{L}_{j'_d}\left(u\right) du={\bf 1}_{{\bf j}={\bf j'}}.
\end{eqnarray*}
Moreover, by \eqref{eq:hilf005}
\begin{eqnarray*}
 |(II)|&\leq&  \frac{\varpi(1/N)}{f_*} e^{2Q}.
\end{eqnarray*}
Hence, by Gershgorin's theorem (see \cite{HJ}, Theorem 6.1.1),
$$
\lambda_{\max}(R_{\bf i})\leq f(x_{\bf i})\left(1+  {{D+Q}\choose{D}} \frac{\varpi(1/N)}{f_*} e^{2Q} \right) \leq \frac{3}{2} f^*,
$$
if condition \eqref{eq:GC} is satisfied. In the same way we get the bound $\lambda_{\min}(R_{\bf i})\geq f_*/2$.
\end{proof}

\begin{proof}[Proof of Theorem \ref{thm:polynomials2}]
 In view of the previous lemma, we may apply Theorem \ref{thm:noiseless} with
$$
\lambda^*=\frac{3f^*}{2},\quad \lambda_*=\frac{f_*}{2},\quad m\leq  N^D e^{2Q}  {{D+Q}\choose{D}}.
$$
Then,
\begin{eqnarray*}
 -\frac{3 \epsilon^2 L}{(36\frac{f^*}{f_*^2}+4\frac{\epsilon}{f_*})N^D e^{2Q}  {{D+Q}\choose{D}}+6\epsilon \frac{f^*}{f_*}}  \leq -\log(c_0 N^{D+2(Q+\gamma)}).
\end{eqnarray*}
Taking into account that $K=N^D{{D+Q}\choose{D}}$, Theorem \ref{thm:noiseless} yields,
\begin{eqnarray*}
 && E\left[\int_{[0,1]^D} |y(x)-\hat y_L(x)|^2 \mu(dx) \right]\\ &\leq& \left(1+ \frac{3f^*}{f_*(1-\epsilon)}\right) \inf_{\alpha\in \R^K} \int_{[0,1]^D} |y(x)-\alpha^\top \eta(x)|^2 \mu(dx)  \\
&&+   N^{-2(Q+\gamma)} \frac{2 {{D+Q}\choose{D}}}{c_0} \int_{[0,1]^D} |y(x)|^2 \mu(dx).
\end{eqnarray*}
It remains to estimate the approximation error due to the basis choice. Note that 
$$
\inf_{\alpha\in \R^K} \int_{[0,1]^D} |y(x)-\alpha^\top \eta(x)|^2 \mu(dx) = \sum_{{\bf i}} \inf_{p_Q} \int_{C_{\bf i}} \left|y(x)-p_Q(x) \right|^2 \mu(dx),
$$
where the infimum runs over the polynomials of degree at most $Q$. On each cube $C_{\bf i}$ 
fix some $x_{\bf i}\in C_{\bf i}$  and denote by $p_Q$ the Taylor polynomial of degree $Q$ around $x_{\bf i}$. Then, for $x\in  C_{\bf i}$ there is a point $\zeta$ on the line connecting 
$x_{\bf i}$ and $x$ such that
\begin{eqnarray*}
 y(x)-p_Q(x)=\frac{1}{Q!} \sum_{(d_1,\ldots, d_Q)\in \{1,\ldots,D\}^Q} \left(\frac{\partial^Q  y(\zeta)}{\partial_{x_{d_1},\ldots, x_{d_Q}}}- \frac{\partial^Q  y(x)}{\partial_{x_{d_1},\ldots, x_{d_Q}}}\right) \prod_{j=1}^Q (x_{d_j}-x_{{\bf i}, d_j})
\end{eqnarray*}
Hence,
\begin{eqnarray*}
| y(x)-p_Q(x)|\leq \frac{D^Q}{Q!} C |x-\zeta|_2^\gamma \,  |x-x_{\bf i}|_\infty^Q \leq  \frac{D^Q}{Q!} C \sqrt{D}^\gamma N^{-(Q+\gamma)},
\end{eqnarray*}
which implies
\begin{eqnarray*}
 \inf_{\alpha\in \R^K} \int_{[0,1]^D} |y(x)-\alpha^\top \eta(x)|^2 \mu(dx)  \leq  C^2\frac{D^{2Q+\gamma}}{(Q!)^2} N^{-2(Q+\gamma)}.
\end{eqnarray*}
\end{proof}

\section{Proof of Theorem \ref{thm:main}}\label{sec:proof}

This section is devoted to the proof of Theorem \ref{thm:main}. Thus, throughout this section all the notation introduced in Algorithm \ref{alg:1} and Theorem \ref{thm:main} is in force. We define 
\begin{eqnarray*}
 X^{(\Delta)}_2&=&X_2=X_1+b(X_1)\Delta+\sigma(X_1)\sqrt{\Delta} \xi, \\
X^{(\Delta,r_2)}_2&=&X_1+b(X_1)\Delta+\sigma(X_1)\sqrt{\Delta} [\xi]_{r_2}, \\
X^{(\Delta,r_2,\ii)}_2&=&U_\ii+b(U_\ii)\Delta+\sigma(U_\ii)\sqrt{\Delta} [\xi]_{r_2}, 
\end{eqnarray*}
where, for $\ii\in I_\Delta$, $U_\ii$ is uniformly distributed on the cube $\Gamma_\ii=\prod_{d=1}^D (hi_d,h(i_d+1)]$. Here, $h=c_{\textnormal{cube}}\Delta^{\gamma_{\textnormal{cube}}}$. 

The first lemma estimates the influence of truncating the Gaussian innovations at level $r_2$, where
$$
r_2=\sqrt{2\log(c_{2,\textnormal{trunc}}\,\Delta^{- \gamma_{2,\textnormal{trunc}}} \log(\Delta^{-1}))}.
$$

\begin{lem}[Truncation of Gaussian innovations]\label{lem:trunc1}
 Suppose $y\in \mathcal{C}^1_b(\R^D)$ and $\Delta <e^{-1}$, $r_2\geq 1$. Then, there is a constant $C_1$  such that
\begin{eqnarray*}
 E[|\mathcal{H}_{\iota,\Delta}(\xi)  y(X^{(\Delta)}_2)-\mathcal{H}_{\iota,\Delta}([\xi]_{r_2})  y(X^{(\Delta,r_2)}_2)|^2|X_1]\leq C_1 \Delta^{2\gamma_{2,\textnormal{trunc}}/3-|\iota|_1}
\end{eqnarray*}
 \end{lem}
\begin{proof}
 Let
$$
g(x)=g(x;X_1)=\mathcal{H}_{\iota,\Delta}(x)y(X_1+b(X_1)\Delta+\sigma(X_1)\sqrt{\Delta} x).
$$
Then,
\begin{eqnarray}\label{eq:hilf007}
&& |\mathcal{H}_{\iota,\Delta}(\xi)  y(X^{(\Delta)}_2)-\mathcal{H}_{\iota,\Delta}([\xi]_{r_2})  y(X^{(\Delta,r_2)}_2)|= |g(\xi)-g([\xi]_{r_2}) | \nonumber \\ &=& |\int_0^1 \langle \nabla g(\xi+u([\xi]_{r_2}-\xi)), [\xi]_{r_2}-\xi \rangle du|
\nonumber \\ &\leq&   \int_0^1  |[\xi]_{r_2}-\xi|_2 \,|\nabla g(\xi+u([\xi]_{r_2}-\xi))|_2 du.
\end{eqnarray}
We next estimate the gradient of $g$. By the product rule,
$$
|\nabla g(x)|_2\leq \|y\|_\infty \, |\nabla \mathcal{H}_{\iota,\Delta}(x)|_2+\sqrt{\Delta}C_{b,\sigma}\|\nabla y\|_\infty \, |\mathcal{H}_{\iota,\Delta}(x)|.
$$
Denote by $C_{\mathcal{H},q}$ a positive constant such that
$$
|\mathcal{H}_{\iota_0,1}(x)|\leq C_{\mathcal{H},q}(1+|x|_2^q)
$$
for every multi-index $\iota_0$ satisfying $|\iota_0|_1\leq q$. As, for $\iota_j\geq 1$,
$$
\frac{\partial}{\partial x_j}\mathcal{H}_{\iota,\Delta}(x)= \Delta^{-|\iota|_1/2}  \iota_{j} \mathcal{H}_{\iota_j-1}(x_j) \prod_{d\neq j}  \mathcal{H}_{\iota_d}(x_d)=\Delta^{-|\iota|_1/2}  \iota_{j} \mathcal{H}_{\iota-e_j,1}(x),
$$
($e_j$ denoting the $j$th unit vector in $\R^D$), we obtain for $\Delta\leq 1$
\begin{eqnarray*}
|\nabla g(x)|_2&\leq& (\|y\|_\infty \sqrt{D} |\iota|_\infty +C_{b,\sigma}\|\nabla y\|_\infty ) C_{\mathcal{H},|\iota|_1}(1+|x|_2^{|\iota|_1} )\Delta^{-|\iota|_1/2}\\&=:& \hat C_1 (1+|x|_2^{|\iota|_1} )\Delta^{-|\iota|_1/2}.
\end{eqnarray*}
Plugging this estimate into \eqref{eq:hilf007} and applying Jensen's inequality,  Fubini's theorem, and H\"older's inequality yields
\begin{eqnarray*}
 && E[|\mathcal{H}_{\iota,\Delta}(\xi)  y(X^{(\Delta)}_2)-\mathcal{H}_{\iota,\Delta}([\xi]_{r_2})  y(X^{(\Delta,r_2)}_2)|^2|X_1]  \\
&\leq & E[|[\xi]_{r_2}-\xi|_2^3 ]^{2/3} \int_0^1 E[ (\hat C_1 (1+|\xi+u([\xi]_{r_2}-\xi)|_2^{|\iota|_1} )\Delta^{-|\iota|_1/2} )^6 ]^{1/3} du \\
&\leq& \hat C_1^2 \Delta^{-|\iota|_1}  E[|[\xi]_{r_2}-\xi|_2^3 ]^{2/3} \max_{u\in [0,1]} E[(1+|\xi+u([\xi]_{r_2}-\xi)|_2^{|\iota|_1} )^6 ]^{1/3}. 
\end{eqnarray*}
We next estimate the last factor. By Jensen's inequality (twice),
\begin{eqnarray*}
 E[(1+|\xi+u([\xi]_{r_2}-\xi)|_2^{|\iota|_1} )^6 ] \leq 2^{5}(1+ D^{3 |\iota|_1}E[|\xi_1+u([\xi_1]_{r_2}-\xi_1)|^{6 |\iota|_1} ]
\end{eqnarray*}
By convexity, the supremum over $u\in [0,1]$ of the righthand side is attained at $u=0$ or $u=1$. As the absolute moments of the truncated normal distribution are smaller than those of the normal distribution, it is, in fact, attained at $u=0$. 
Recalling that the $q$th moment of a normal distribution is given by $(q-1)!!$ for even $q$, we thus obtain
\begin{eqnarray*}
 && E[|\mathcal{H}_{\iota,\Delta}(\xi)  y(X^{(\Delta)}_2)-\mathcal{H}_{\iota,\Delta}([\xi]_{r_2})  y(X^{(\Delta,r_2)}_2)|^2|X_1]  \\
&\leq& \hat C_1^2 2^{5/3} (1+ D^{3 |\iota|_1+1}(6 |\iota|_1-1)!!)^{1/3}\Delta^{-|\iota|_1}  E[|[\xi]_{r_2}-\xi|_2^3 ]^{2/3}. 
\end{eqnarray*}
Finally,
\begin{eqnarray*}
  E[|[\xi]_{r_2}-\xi|_2^3 ]&=&E[(\sum_{d=1}^D |\xi_d-[\xi_d]_{r_2}|^2 ])^{3/2}]\leq D^{3/2}  E[|\xi_1-[\xi_1]_{r_2}|^3]\\ &\leq& 2D^{3/2} \int_{r_2}^\infty u^3\varphi(u)du=
D^{3/2} (2r^2_2\varphi(r_2)+4\varphi(r_2))\leq 6 D^{3/2} r^2_2\varphi(r_2).
\end{eqnarray*}
Combining the previous estimates, we arrive at
\begin{eqnarray*}
  && E[|\mathcal{H}_{\iota,\Delta}(\xi)  y(X^{(\Delta)}_2)-\mathcal{H}_{\iota,\Delta}([\xi]_{r_2})  y(X^{(\Delta,r_2)}_2)|^2|X_1]\\ &\leq& 2^{5/3} 6^{2/3} D
(\|y\|_\infty \sqrt{D} |\iota|_\infty +C_{b,\sigma}\|\nabla y\|_\infty )^2 C^2_{\mathcal{H},|\iota|_1}  (1+ D^{3 |\iota|_1+1}(6 |\iota|_1-1)!!)^{1/3} \\ &&\times \Delta^{-|\iota|_1}  (r^2_2\varphi(r_2))^{2/3}.
\end{eqnarray*}
Taking the form of $r_2$ into account, we get
$$
r^2_2\varphi(r_2)\leq \sqrt{\frac{2}{\pi}} (\log(c_{2,\textnormal{trunc}})+\gamma_{2,\textnormal{trunc}}+1)\frac{\Delta^{\gamma_{2,\textnormal{trunc}}}}{c_{2,\textnormal{trunc}}},
$$
which finishes the proof.
\end{proof}

We next denote by $\Theta$ a (finite) family of random variables independent of $(X_1,\xi,U_i)_{i\in I_\Delta}$. We think of $\Theta$   as containing the simulated samples which are applied for estimating $y$, and assume that some measurable estimator 
$\hat y(x_1, x_2;\theta)$ of $y(x_2)$ is given. Recall that the truncation in space for the $x_1$-variables (i.e. the set $\Gamma$ in Step 1 of Algorithm \ref{alg:1}) depends on the constant
$$r_1=\sqrt{C_{2,f} \chi^2_D(c_{1,\textnormal{trunc}}\,\Delta^{ \gamma_{1,\textnormal{trunc}}})}.$$

The next lemma takes care of the change in the sampling distribution and removes the derivative weight.

\begin{lem}[Removal of the derivative weight and change of measure]\label{lem:trunc2} Suppose $y\in \mathcal{C}^{Q}_b(\R^D)$ for $Q\in \N$ s.t. $Q\geq|\iota|_1 $, and let $\Delta <\min(e^{-1}, c_{1,\textnormal{trunc}}^{-1/\gamma_{1,\textnormal{trunc}}})$,
 $r_2\geq 1$, $\gamma_{2,\textnormal{trunc}}\geq 3|\iota|_1/2$.
 Let $\Gamma:=\Gamma^{(\Delta)}:=\cup_{\ii\in I_\Delta} \Gamma_\ii$ and assume $\hat y(x_1,x_2;\theta)=0$ for $x_1\notin \Gamma$.   
Choose $\sigma(\Theta)$-measurable sets $\tilde \Gamma_\ii$ $(\ii\in I_\Delta)$ such that for every $\ii \in I_\Delta$,
 $$
 \hat y(x_1,x_2;\Theta(\omega))=0 \textnormal{ for every } x_1\in \Gamma_\ii, \;x_2\in \R^D,\; \omega \in \tilde \Gamma_\ii^c.
 $$
 (One may choose $\tilde \Gamma_\ii=\Omega$ to ensure that this condition is always satisfied).
 
  Then, there is a constant $C_2$ such that
\begin{eqnarray*}
 && E[|E[\mathcal{H}_{\iota,\Delta}([\xi]_{r_2})  y(X^{(\Delta,r_2)}_2)|X_1] -E[\mathcal{H}_{\iota,\Delta}([\xi]_{r_2})  \hat y(X_1,X^{(\Delta,r_2)}_2,\Theta)|X_1,\Theta]|^2 ]\\ &\leq& C_2\,\left(\Delta^{\gamma_{1,\textnormal{trunc}}} 
+ (1+\log(\Delta^{-1})^{D/2}) \max_{\ii\in I_\Delta}  P(\tilde \Gamma_\ii^c)\right. \\ && + \left.   \Delta^{-|\iota|_1} (1+\log(\Delta^{-1})^{D/2}) \max_{\ii\in I_\Delta} 
E[{\bf 1}_{\tilde \Gamma_i}| y(X^{(\Delta,r_2,\ii)}_2) - \hat y(U_\ii,X^{(\Delta,r_2,\ii)}_2,\Theta) |^2]\right).
\end{eqnarray*}
\end{lem}
\begin{proof}
 We decompose, using H\"older's inequality and exploiting that $\Theta$ is independent of $(X_1,X_2,\xi)$,
\begin{eqnarray*}
 && |E[\mathcal{H}_{\iota,\Delta}([\xi]_{r_2})  y(X^{(\Delta,r_2)}_2)|X_1] -E[\mathcal{H}_{\iota,\Delta}([\xi]_{r_2})  \hat y(X_1,X^{(\Delta,r_2)}_2,\Theta)|X_1,\Theta]|^2 \\
&\leq& \left({\bf 1}_{\{X_1\notin \Gamma\}}+\sum_{\ii\in I_\Delta} {\bf 1}_{\tilde\Gamma_\ii^c}{\bf 1}_{\{X_1\in \Gamma_\ii\}}  \right) E[\mathcal{H}_{\iota,\Delta}([\xi]_{r_2})  y(X^{(\Delta,r_2)}_2)|X_1]^2 \\ 
&&+ \sum_{\ii\in I_\Delta}  {\bf 1}_{\tilde \Gamma_\ii}{\bf 1}_{\{X_1\in \Gamma_\ii\}}  E[|\mathcal{H}_{\iota,\Delta}([\xi]_{r_2})|^2] E[   |y(X^{(\Delta,r_2)}_2)-\hat y(X_1,X^{(\Delta,r_2)}_2,\Theta) |^2|X_1,\Theta] 
\\ &=&(I)+(II).
\end{eqnarray*}
Let
$$
z(x)=E[\mathcal{H}_{\iota,\Delta}(\xi)  y(X^{(\Delta)}_2)|X_1=x].
$$
By \eqref{eq:relation_weight_derivative}, the regression function $z$ is bounded. Applying the previous lemma, we get,
\begin{eqnarray*}
 (I)\leq \left({\bf 1}_{\{X_1\notin \Gamma\}}+\sum_{\ii\in I_\Delta} {\bf 1}_{\tilde \Gamma_\ii^c}{\bf 1}_{\{X_1\in \Gamma\ii\}}  \right) 2 (\|z\|_\infty^2+  C_1 \Delta^{2\gamma_{2,\textnormal{trunc}}/3-|\iota|_1}), 
\end{eqnarray*}
and, thus, by independence of $X_1$ and $\Theta$,
\begin{equation}\label{eq:hilf009}
E[(I)]\leq 2 \left(P(\{X_1\notin \Gamma\})+ P(\{X_1\in \Gamma\})\max_{i\in I_\Delta} P(\tilde \Gamma_i^c) \right)  (\|z\|_\infty^2+  C_1 ).
\end{equation}
However,
\begin{eqnarray*}
 P(\{X_1\notin \Gamma\})&=&\int_{\Gamma^c} f(x) dx \leq \int_{\{|x|_2>r_1\}}   \frac{C_{1,f}}{(2\pi C_{2,f})^{D/2}} \exp \left\{\frac{-|x|_2^2}{2 C_{2,f}} \right\} dx
 \\ &=& C_{1,f}  \int_{\{|x|^2_2>r_1^2/C_{2,f}\}}   \frac{1}{(2\pi)^{D/2}} \exp \left\{\frac{-|x|_2^2}{2 } \right\} dx\\ &=&C_{1,f} P(\{\Xi>r_1^2/C_{2,f}\}),
\end{eqnarray*}
where $\Xi$ is $\chi^2$-distributed with $D$ degrees of freedom. Taking the particular form of $r_1$ into account, we obtain 
\begin{equation}\label{eq:hilf010}
 P(\{X_1\notin \Gamma\}) \leq  C_{1,f} c_{1,\textnormal{trunc}}\,  \Delta^{ \gamma_{1,\textnormal{trunc}}}.
\end{equation}
Moreover,
\begin{equation}\label{eq:hilf011}
P(\{X_1\in \Gamma\})\leq \frac{C_{1,f}}{(2\pi C_{2,f})^{D/2}} \lambda^{\otimes D}(\Gamma).
\end{equation}
Since $\Gamma \subset [-(r_1+h),r_1+h]^D$ we obtain,
$$
\lambda^{\otimes D}(\Gamma)\leq 2^D(r_1+h)^D. 
$$
As the $(1-\alpha)$-quantiles of a $\chi^2$-distribution with $D$ degrees of freedom satisfy 
$$
\chi^2_D(\alpha)\leq D+2 \log(1/\alpha) + 2\sqrt{D\log(1/\alpha)}\leq 2D+3\log(1/\alpha),
$$
see e.g. \cite{In},
we observe that 
$$
\chi^2_D(c_{1,\textnormal{trunc}}\,\Delta^{ \gamma_{1,\textnormal{trunc}}})\leq 2D+ 3\log(c_{1,\textnormal{trunc}}^{-1})+3\gamma_{1,\textnormal{trunc}} \log(\Delta^{-1}).
$$
Hence,
\begin{equation}\label{eq:hilf012}
\lambda^{\otimes D}(\Gamma)\leq 2^D (c_{\textnormal{cube}} + \sqrt{2C_{2,f}D+ 3C_{2,f}(\log(c_{1,\textnormal{trunc}}^{-1})\vee 0)}+\sqrt{3C_{2,f}\gamma_{1,\textnormal{trunc}} \log(\Delta^{-1})})^D.
\end{equation}
Combining \eqref{eq:hilf009}--\eqref{eq:hilf012}, we arrive at
$$
E[(I)] \leq C_2 \left(\Delta^{ \gamma_{1,\textnormal{trunc}}}+ (1+\log(\Delta^{-1})^{D/2}) \max_{\ii\in I_\Delta} P(\tilde \Gamma_i^c) \right).
$$

We next turn to term (II). Applying the previous lemma once more with $y\equiv 1$, we get by the orthonormality of the Hermite polynomials $\mathcal{H}_{\iota,1}$ with respect to the distribution of $\xi$
$$
E[|\mathcal{H}_{\iota,\Delta}([\xi]_{r_2})|^2]\leq 2(E[|\mathcal{H}_{\iota,\Delta}(\xi)|^2]+  C_1 \Delta^{2\gamma_{2,\textnormal{trunc}}/3-|\iota|_1}) \leq 2(1+C_1) \Delta^{-|\iota|_1}.
$$
Thus,
\begin{eqnarray*}
 && E[(II)]\\&\leq & 2(1+C_1) \Delta^{-|\iota|_1} \sum_{\ii\in I_{\Delta}} E[{\bf 1}_{\tilde\Gamma_\ii} \int_{\R^D} \int_{\Gamma_\ii} |y(x_1+b(x_1)\Delta+\sigma(x_1)\sqrt{\Delta}[w]_{r_2})\\ &&\quad\quad- \hat y(x_1,x_1+b(x_1)\Delta+\sigma(x_1)\sqrt{\Delta}[w]_{r_2};\Theta)|^2
f(x_1)\varphi^{\otimes D}(w) dx_1 dw] \\
&\leq&  \Delta^{-|\iota|_1} \frac{2(1+C_1) C_{1,f}}{(2\pi C_{2,f})^{D/2}} \lambda^{\otimes D}(\Gamma) \max_{\ii\in I_{\Delta}} E[ {\bf 1}_{\tilde\Gamma_\ii} \int_{\R^D} \int_{\Gamma_\ii} |y(x_1+b(x_1)\Delta+\sigma(x_1)\sqrt{\Delta}[w]_{r_2})\\ &&\quad\quad- \hat y(x_1,x_1+b(x_1)\Delta+\sigma(x_1)\sqrt{\Delta}[w]_{r_2};\Theta)|^2
\frac{1}{\lambda^{\otimes D}(\Gamma_\ii)}\varphi^{\otimes D}(w) dx_1 dw] \\ &\leq &  \frac{2(1+C_1)C_{1,f}}{(2\pi C_{2,f})^{D/2}} \lambda^{\otimes D}(\Gamma)  \Delta^{-|\iota|_1} \max_{\ii\in I_\Delta} 
E[{\bf 1}_{\tilde\Gamma_\ii} | y(X^{(\Delta,r_2,\ii)}_2) - \hat y(U_\ii,X^{(\Delta,r_2,\ii)}_2,\Theta) |^2].
\end{eqnarray*}
In view of \eqref{eq:hilf012}, the proof is complete.
\end{proof}

We next consider
$$
\tilde X_\Delta=U_h+b(U_h)\Delta +\sigma(U_h)\sqrt{\Delta} [\xi]_{r_2},
$$
where $U_h$ is uniformly distributed on a cube $\tilde \Gamma=a_0+(-h/2,h/2]^D$, $a_0\in \R^D$, and $\xi$ is a $D$-dimensional vector of independent standard normal variables. We suppose that $h=c_{\textnormal{cube}} \Delta^{\gamma_{\textnormal{cube}}}$ for some $0<\gamma_{\textnormal{cube}}< 1/2$ 
and $c_{\textnormal{cube}}>0$. Hence, $\tilde X_\Delta$ corresponds to $X^{(\Delta,r_2,{\bf i})}_2$ on a `generic' cube of volume $h^D$.

 Recall the construction of  the basis functions in Algorithm \ref{alg:1} for fixed  $Q\in \N_0$: Denoting by $\mathcal{L}_q$ the Legendre polynomial of degree $q\leq Q$,
and, given 
a multi-index ${\bf j}\in \N_0^D$  such that $|\jj|_1\leq Q$, we consider the polynomials
$$
p_{\bf j}(x)=  \prod_{d=1}^D \sqrt{2j_d+1} \, \mathcal{L}_{j_d}\left(x_d\right),
$$
which are finally rescaled to 
$$
\eta_{\bf j}(x)=p_{\bf j}\left(\frac{x-a_0}{h/2}\right),\quad {\bf j}\in \N_0^D,\, \sum_d j_d\leq Q.
$$
We fix some ordering of these ${{D+Q}\choose{D}}$  basis functions and write $\eta_k$, $k=1,\ldots K:= {{D+Q}\choose{D}}$,
$$
R_{\Delta}=(E[\eta_k(\tilde X_\Delta)\eta_\kappa(\tilde X_\Delta)])_{k,\kappa=1,\ldots,K}.
$$
In order to apply Theorem \ref{thm:noiseless}, we need to estimate
\begin{eqnarray*}
 \lambda_{\min,\Delta}:=\lambda_{\min}(R_{\Delta}),\quad \lambda_{\max,\Delta}:=\lambda_{\max}(R_{\Delta}),\quad m_{\Delta}:=\sup_{x\in \supp(\tilde X_\Delta))} \sum_{k=1}^K |\eta_k(x)|^2.
\end{eqnarray*}

\begin{lem}\label{lem:eigenvalues}
 As $\Delta\rightarrow 0$, we have
$$
\lambda_{\min,\Delta} \rightarrow 1,\quad \lambda_{\max,\Delta}\rightarrow 1,\quad m_{\Delta}\rightarrow \sum_{{\bf j}\in \N_0^D; |{\bf j}|_1\leq Q}\; \prod_{d=1}^D (2j_d+1).
$$
\end{lem}
As a preparation for the eigenvalue estimates we first prove the following lemma.
\begin{lem}
 Suppose $\tilde\eta=(\tilde\eta_1,\ldots, \tilde\eta_K)^\top$  is a system of orthonormal polynomials of degree at most $Q\in \N$ for the law of some $\R^D$-valued random variable  $ X^{(1)}$. Let $X^{(2)}$ be another $\R^D$-valued random variable. We consider the matrix 
$$
R^{(2)}= E[\tilde\eta(X^{(2)})\tilde\eta^\top(X^{(2)})].
$$
Then, there is a constant $c_{\tilde\eta}$, which only depends on the coefficients of the polynomials $\tilde\eta_k$ (and on $Q$, $D$) such that
\begin{eqnarray*}
\lambda_{\min}(R^{(2)}) &\geq& 1- K c_{\tilde\eta}  \left( E[|  X^{(1)}-X^{(2)} |^2_2]^{1/2}(1+  E[|  X^{(1)}|^{4Q-2}_2]^{1/2})+  E[|  X^{(1)}-X^{(2)} |^{2Q}_2]\right)\\
\lambda_{\max}(R^{(2)}) &\leq& 1+ K c_{\tilde\eta}  \left( E[|  X^{(1)}-X^{(2)} |^2_2]^{1/2}(1+  E[|  X^{(1)}|^{4Q-2}_2]^{1/2})+  E[|  X^{(1)}-X^{(2)} |^{2Q}_2]\right) 
\end{eqnarray*}
\end{lem}
\begin{proof}
 Recall that $R^{(1)}= E[\tilde\eta(X^{(1)})\tilde\eta^\top(X^{(1)})]$ is the  identity matrix. In view of Gershgorin's theorem (\cite{HJ}, Theorem 6.1.1), it hence suffices to show that for every $k,k'=1,\ldots K$,
\begin{eqnarray*}
 &&| E[\tilde\eta_k(X^{(2)})\tilde\eta_{k'}(X^{(2)})]-E[\tilde\eta_k(X^{(1)})\tilde\eta_{k'}(X^{(1)})]|\\ &\leq& c_{\tilde\eta}  \left( E[|  X^{(1)}-X^{(2)} |^2_2]^{1/2}(1+  E[|  X^{(1)}|^{4Q-2}_2]^{1/2})+  E[|  X^{(1)}-X^{(2)} |^{2Q}_2]\right) 
\end{eqnarray*}
As $\nabla(\tilde\eta_{k}\tilde\eta_{k'})$ is a vector of polynomials of degree at most $2Q-1$ whose coefficients only depend on the coefficients of the polynomials in the system $\tilde\eta$, there is a constant $c_{\tilde\eta}$ (depending also on $Q$, $D$) such that
$$
|\nabla(\tilde\eta_{k}\tilde\eta_{k'})(x)|_2\leq c_{\tilde\eta} (1+|x|_2^{2Q-1})
$$ 
for every $k,k'=1,\ldots K$. Hence, by Jensen's and H\"older's inequality,
\begin{eqnarray*}
 && | E[\tilde\eta_k(X^{(2)})\tilde\eta_{k'}(X^{(2)})]-E[\tilde\eta_k(X^{(1)})\tilde\eta_{k'}(X^{(1)})]|\\
&=&  |\int_0^1    E[(X^{(2)}-X^{(1)})^\top \nabla(\eta_{k}\eta_{k'})(X^{(1)}+u(X^{(2)}-X^{(1)}))] du| 
\\ &\leq &   c_{\tilde\eta} E[|(X^{(2)}-X^{(1)})|_2 \,(1+(|X^{(1)}|_2+|X^{(2)}-X^{(1)}|_2)^{2Q-1})]
\\ &\leq & c_{\tilde\eta} 2^{2Q-2}  E[|(X^{(2)}-X^{(1)})|^2_2]^{1/2} (1+ E[|X^{(1)}|_2^{4Q-2}]^{1/2})+c_{\tilde\eta}2^{2Q-2} E[|(X^{(2)}-X^{(1)})|_2^{2Q}].
\end{eqnarray*}

\end{proof}

\begin{proof}[Proof of Lemma \ref{lem:eigenvalues}]
 For the eigenvalue estimates we apply the previous lemma with
$$
X^{(1)}=\frac{U_h-a_0}{h/2},\quad X^{(2)}=\frac{\tilde X_{\Delta}-a_0}{h/2}
$$
to the system of multivariate Legendre polynomials $(p_{\bf j})_{|{\bf j}|_1\leq Q}$. Since $X^{(1)}$ is uniformly distributed on the cube $(-1,1]^D$, we obtain,
\begin{eqnarray*}
 E[ p_{\bf j}(X^{(1)})p_{{\bf j}'}(X^{(1)})]=\prod_{d=1}^D \int_{-1}^1 \sqrt{\frac{2j_d+1}{2}} \sqrt{\frac{2j'_d+1}{2}}\mathcal{L}_{j_d'}(x_d)\mathcal{L}_{j_d}(x_d) dx_d ={\bf 1}_{j=j'}.
\end{eqnarray*}
Hence, $(p_{\bf j})_{|{\bf j}|_1\leq Q}$ are orthonormal with respect to the law of $X^{(1)}$. Moreover, 
\begin{eqnarray*}
 |X^{(2)}-X^{(1)}|_2&=&  2|b(U_h)c_{\textnormal{cube}}^{-1}\Delta^{1-\gamma_{\textnormal{cube}}}+ c_{\textnormal{cube}}^{-1} \sigma(U_h) \Delta^{1/2-\gamma_{\textnormal{cube}}} [\xi]_{r_2}|_2 
\\ &\leq &   \frac{2c_{b,\sigma}}{c_{\textnormal{cube}}} (\Delta^{1-\gamma_{\textnormal{cube}}} + \Delta^{1/2-\gamma_{\textnormal{cube}}} |\xi|_2).
\end{eqnarray*}
As $\gamma_{\textnormal{cube}}=\frac{\rho+|\iota|_1}{2(Q+1)}<\frac{1}{2}$, we observe that
$$
E[|X^{(2)}-X^{(1)}|_2^2 ]+  E[|X^{(2)}-X^{(1)}|_2^{2Q} ]\rightarrow 0 
$$
for $\Delta\rightarrow 0$, and, consequently, by the previous lemma, $\lambda_{\min,\Delta},\,\lambda_{\max,\Delta}\rightarrow 1$.

We now turn to the limiting behavior of $m_\Delta$. Note first that
$$
m_{\Delta}=\sup_{x\in \supp\left(X^{(2)}\right)} \sum_{|{\bf j}|_1\leq Q} \prod_{d=1}^D (2j_d+1) \, |\mathcal{L}_{j_d}\left(x_d\right)|^2.$$
As
$$
|X^{(2)}-X^{(1)}|_\infty \leq  \frac{2c_{b,\sigma} }{c_{\textnormal{cube}}} (\Delta^{1-\gamma_{\textnormal{cube}}} + \sqrt{D} \Delta^{1/2-\gamma_{\textnormal{cube}}} r_2),
$$
and 
$X^{(1)}$ is uniformly distributed on $(-1,1]^D$, the support of $X^{(2)}$ is contained in the cube 
\begin{eqnarray*}
\left[-\left(1+ \frac{2c_{b,\sigma} }{c_{\textnormal{cube}}} (\Delta^{1-\gamma_{\textnormal{cube}}} + \sqrt{D} \Delta^{1/2-\gamma_{\textnormal{cube}}} r_2)\right), 
1+\frac{2c_{b,\sigma} }{c_{\textnormal{cube}}} (\Delta^{1-\gamma_{\textnormal{cube}}} + \sqrt{D} \Delta^{1/2-\gamma_{\textnormal{cube}}} r_2)\right]^D.
\end{eqnarray*}
We now recall that the squared univariate Legendre polynomials achieve their maximum on $[-u,u]$ at $u$, if $u\geq 1$,
Hence, 
$$
m_{\Delta}\leq \sum_{|{\bf j}|_1\leq Q} \prod_{d=1}^D (2j_d+1) \, \left|\mathcal{L}_{j_d}\left( 1+\frac{2c_{b,\sigma} }{c_{\textnormal{cube}}} (\Delta^{1-\gamma_{\textnormal{cube}}} + \sqrt{D} \Delta^{1/2-\gamma_{\textnormal{cube}}} r_2)\right)\right|^2
$$
As $(1,\ldots,1) \in  \supp (X^{(2)})$, we obtain 
$$
m_\Delta \geq  \sum_{|{\bf j}|_1\leq Q} \prod_{d=1}^D (2j_d+1) \, |\mathcal{L}_{j_d}\left(1\right)|^2
$$
Consequently,
$$
\lim_{\Delta\rightarrow 0} m_{\Delta}= \sum_{{\bf j}\in \N_0^D; |{\bf j}|_1\leq Q}\; \prod_{d=1}^D (2j_d+1),
$$
because the Legendre polynomials take value one at one and $\gamma_{\textnormal{cube}}<1/2$.
\end{proof}

In view of Theorem \ref{thm:noiseless}, the following lemma is the key to control the statistical error in Algorithm \ref{alg:1}.

\begin{lem}\label{lem:statistical}
 Let 
$$
L=L_\Delta=\lceil \gamma_{\textnormal{paths}}\, c_{1,\textnormal{paths}} \log(c_{2,\textnormal{paths}}\; \Delta^{-1})\rceil
$$
for constants
 $$
\gamma_{\textnormal{paths}}>0,\quad c_{2,\textnormal{paths}}>0,\quad c_{1,\textnormal{paths}}> c^*_{\textnormal{paths}}(Q,D):=\frac{2}{3}+\frac{8}{3}\sum_{{\bf j}\in \N_0^D; |{\bf j}|_1\leq Q}\; \prod_{d=1}^D (2j_d+1).
$$
Let 
$$
\tau\in \left(0, \quad 1-\left(\frac{c^*_{\textnormal{paths}}(Q,D)}{c_{1,\textnormal{paths}}}\right)^{1/2} \right).
$$
Then, there is a constant $\Delta_0>0$ such that, for every $\Delta\leq \Delta_0$, $\epsilon_\Delta:=1-\tau/\lambda_{\min,\Delta}\in (0,1-\tau/2]$ and
$$
2K\exp\left\{-\frac{3 \epsilon_\Delta^2 L}{6 m_\Delta\lambda_{\max,\Delta} /\lambda_{\min,\Delta}^2+ 2\epsilon_\Delta(m_\Delta/\lambda_{\min,\Delta} +\lambda_{\max,\Delta} /\lambda_{\min,\Delta} )} \right\}
\leq \frac{2 {{D+Q}\choose{D}}}{c_{2,\textnormal{paths}}^{\gamma_{\textnormal{paths}}}}\Delta^{\gamma_{\textnormal{paths}}}.
$$
\end{lem}
\begin{proof}
 We denote by $\beta<1$ the unique constant such that
$$
\tau= \beta\left(1-\left(\frac{c^*_{\textnormal{paths}}(Q,D)}{c_{1,\textnormal{paths}}}\right)^{\beta/2}\right).
$$
Let
$$
m_\infty:=\sum_{{\bf j}\in \N_0^D; |{\bf j}|_1\leq Q}\; \prod_{d=1}^D (2j_d+1).
$$
By Lemma \ref{lem:eigenvalues}, there is a $\Delta_0>0$ such that for every $\Delta\leq \Delta_0$
\begin{eqnarray*}
 \lambda_{\max,\Delta} /\lambda_{\min,\Delta}&\leq&  \left(\frac{c^*_{\textnormal{paths}}(Q,D)}{c_{1,\textnormal{paths}}}\right)^{-(1-\beta)}, \\
 m_\Delta\lambda_{\max,\Delta} /\lambda_{\min,\Delta}^2&\leq & m_\infty \left(\frac{c^*_{\textnormal{paths}}(Q,D)}{c_{1,\textnormal{paths}}}\right)^{-(1-\beta)},\\ 
 m_\Delta/\lambda_{\min,\Delta}&\leq& m_\infty \left(\frac{c^*_{\textnormal{paths}}(Q,D)}{c_{1,\textnormal{paths}}}\right)^{-(1-\beta)}, \\
\lambda_{\min,\Delta}&\in& [\beta,2],
\end{eqnarray*}
since the expressions on the left-hand side converge to $1, m_\infty, m_\infty,1$, respectively and $\left(\frac{c^*_{\textnormal{paths}}(Q,D)}{c_{1,\textnormal{paths}}}\right)^{-(1-\beta)}>1$.
The upper and lower bound on $\lambda_{\min,\Delta}$ ensure that
$$
\epsilon_\Delta \geq \left(\frac{c^*_{\textnormal{paths}}(Q,D)}{c_{1,\textnormal{paths}}}\right)^{\beta/2} >0,
$$
and $\epsilon_\Delta\leq 1-\tau/2$.
Hence,
$$
3\epsilon_\Delta^2 c_{1,\textnormal{paths}} \geq  \left(\frac{c^*_{\textnormal{paths}}(Q,D)}{c_{1,\textnormal{paths}}}\right)^{\beta-1}\left(2+8 m_\infty\right).
$$
Inserting the above inequalities and the definition of $L$ yields
\begin{eqnarray*}
&& \frac{3 \epsilon_\Delta^2 L}{6 m_\Delta\lambda_{\max,\Delta} /\lambda_{\min,\Delta}^2+ 2\epsilon_\Delta(m_\Delta/\lambda_{\min,\Delta} +\lambda_{\max,\Delta} /\lambda_{\min,\Delta} )} \geq \gamma_{\textnormal{paths}}\log(c_{2,\textnormal{paths}}\; \Delta^{-1}). 
\end{eqnarray*}
Recalling that $K={{D+Q}\choose{D}}$, the assertion follows.
\end{proof}

\begin{proof}[Proof of Theorem \ref{thm:main}]
 Write
$$
\hat y(x_1,x_2,\Theta):=\sum_{\ii\in I_\Delta} {\bf 1}_{\Gamma_i}(x_1) \sum_{k=1}^K \alpha_{L,\ii,k} \eta_{i,k}(x_2) 
$$
where the coefficients $\alpha_{L,\ii,k}$ are computed via Algorithm \ref{alg:1} and depend on the simulated samples
$$
\Theta=(U_{\ii,l},\xi_{\ii,l})_{l=1,\ldots,L;\, \ii\in I_{\Delta}}.
$$
Let $(X_1,\xi,U_\ii)_{ \ii\in I_\Delta}$ be an independent family, which is also independent of $\Theta$, and such that $X_1$ is $\mu_1$-distributed, $\xi$ is a vector of length $D$ of independent standard normals, and $U_\ii$ is uniformly distributed 
on $\Gamma_\ii$. Then,
\begin{eqnarray*}
\hat{z}(x)&:=&\sum_{\ii\in I_\Delta} {\bf 1}_{\Gamma_\ii}(x) \sum_{k=1}^K \alpha_{L,\ii,k} E[\eta_{\ii,k}(x+b(x)\Delta+\sigma(x)\sqrt{\Delta} [\xi]_{r_2}) \mathcal{H}_{\iota,\Delta}([\xi]_{r_2})] \\
\end{eqnarray*}
satisfies 
$$
\hat{z}(X_1)= E[\mathcal{H}_{\iota,\Delta}([\xi]_{r_2})  \hat y(X_1,X^{(\Delta,r_2)}_2,\Theta)|X_1,\Theta].
$$
Let 
$
\tilde \Gamma_\ii:=\{s_{\ii,K}^2 \geq \tau L\},
$
where $s_{\ii,K}$ is the smallest singular value of the random regression matrix $A_\ii$.
By Lemmas \ref{lem:trunc1} and \ref{lem:trunc2}, we obtain for sufficiently small $\Delta$,
\begin{eqnarray*}
 &&E\left[\int_{\R^D}  |E[\mathcal{H}_{\iota,\Delta}(\xi)  y(X_2)|X_1=x]-\hat{z}(x)|^2\mu_1(dx) \right]\\ &=&
E\left[  |E[\mathcal{H}_{\iota,\Delta}(\xi)  y(X_2)|X_1]-\hat{z}(X_1)|^2 \right] \\ &\leq & 2 E\left[  |E[\mathcal{H}_{\iota,\Delta}([\xi]_{r_2})  y(X_2^{(\Delta,r_2)})|X_1]-\hat{z}(X_1)|^2 \right] + 2C_1 \Delta^{2\gamma_{2,\textnormal{trunc}}/3-|\iota|_1} \\
&\leq &  2 C_2 \log(\Delta^{-1})^{D/2} \Delta^{-|\iota|_1} \max_{\ii\in I_\Delta} 
E[{\bf 1}_{\{s_{\ii,K}^2 \geq \tau L\}}| y(X^{(\Delta,r_2,i)}_2) - \hat y(U_i,X^{(\Delta,r_2,\ii)}_2,\Theta) |^2] \\ && 
+ 2 C_2  \log(\Delta^{-1})^{D/2} \max_{\ii \in I_\Delta} P(\{s_{\ii,K}^2 \geq \tau L\})
+ 2 C_2 \Delta^{\gamma_{1,\textnormal{trunc}}} + 2C_1 \Delta^{2\gamma_{2,\textnormal{trunc}}/3-|\iota|_1}
.
\end{eqnarray*}
By the choice of $\gamma_{1,\textnormal{trunc}}=\rho$ and $\gamma_{2,\textnormal{trunc}}=1.5(\rho+ |\iota|_1)$, the last two terms are of order $\Delta^\rho$. It thus remains to show that there is a constant 
$C_3\geq 0$ and a $\Delta_0>0$ such that for every $\Delta\leq \Delta_0$ and $\ii\in I_\Delta$ 
\begin{eqnarray}\label{eq:hilf008}
P(\{s_{\ii,K}^2 \geq \tau L\})&\leq& C_3 \Delta^{\rho},\nonumber\\
E[{\bf 1}_{\{s_{\ii,K}^2 \geq \tau L\}}| y(X^{(\Delta,r_2,\ii)}_2) - \hat y(U_\ii,X^{(\Delta,r_2,\ii)}_2,\Theta) |^2]&\leq& C_3 \Delta^{\rho+|\iota|_1}.
\end{eqnarray}
 Note that
$$
\hat y(U_\ii,X^{(\Delta,r_2,\ii)}_2,\Theta) = \sum_{k=1}^K \alpha_{L,\ii,k} \, \eta_{\ii,k}(X^{(\Delta,r_2,\ii)}_2),
$$
where the coefficients $(\alpha_{L,\ii,k})_{k=1,\ldots,K}$ are computed on the $\ii$th cube via Algorithm \ref{alg:2}. 
Let $\epsilon_\Delta:=1-\tau/\lambda_{\min,\Delta}$ 
Applying Lemma \ref{lem:ev_vs_spectral_norm} and  \eqref{eq:hilf002}  in conjunction with Lemma  \ref{lem:statistical}, 
there is a $\Delta_0>0$ such that for $\Delta\leq \Delta_0$
\begin{eqnarray*}
 P(\{s_{\ii,K}^2 \geq \tau L\})&\leq&  \frac{2  \|y\|_\infty^2 {{D+Q}\choose{D}}}{c_{2,\textnormal{paths}}^{\rho}}\Delta^{\rho} 
\end{eqnarray*}
Hence, the first term in \eqref{eq:hilf008} is of the required order as well.

Concerning the second term in \eqref{eq:hilf008}, we obtain in view of \eqref{eq:hilf013},
\begin{eqnarray*}
&& E[{\bf 1}_{\{s_{\ii,K}^2 \geq \tau L\}}| y(X^{(\Delta,r_2,\ii)}_2) - \hat y(U_\ii,X^{(\Delta,r_2,\ii)}_2,\Theta) |^2]\\ &\leq&
 \left(1+ \frac{\lambda_{\max,\Delta}}{\lambda_{\min,\Delta}(1-\epsilon_\Delta)}\right) \inf_{\alpha\in \R^K} E[|y(X^{(\Delta,r_2,\ii)}_2)-\sum_{k=1}^K \alpha_k \eta_{\ii,k}(X^{(\Delta,r_2,\ii)}_2)|^2]
 \end{eqnarray*}
By Lemma \ref{lem:eigenvalues}, we may assume without loss of generality that $\lambda_{\max,\Delta}\leq 2$ for $\Delta\leq \Delta_0$ (by decreasing $\Delta_0$, if necessary). Hence,  for $\Delta\leq \Delta_0$,
$$
\frac{\lambda_{\max,\Delta}}{\lambda_{\min,\Delta}(1-\epsilon_\Delta)} \leq 2\tau^{-1}.
$$
Concerning the approximation error due to the choice of the basis functions, we perform
 a Taylor expansion of order $Q$ around the center $a_\ii$ of the $\ii$th cube. Then, as in the proof of Theorem \ref{thm:polynomials2}, there is a polynomial $\mathcal{P}$ of degree at most $Q$ such that 
$$
|y(X^{(\Delta,r_2,\ii)}_2)- \mathcal{P}(X^{(\Delta,r_2,\ii)}_2)| \leq C_4 |X^{(\Delta,r_2,\ii)}_2-a_\ii|_2^{Q+1}
$$
for some constant $C_4\geq 0$, which depends only on $D$, $Q$, and the $C^{Q+1}_b$-norm of $y$. As this polynomial can be expressed as linear combination of the rescaled Legendre polynomials $\eta_{\ii,k}$, $k=1, \ldots, K$, we observe that 
$$
\inf_{\alpha\in \R^K} E[|y(X^{(\Delta,r_2,\ii)}_2)-\sum_{k=1}^K \alpha_{k} \, \eta_{\ii,k}(X^{(\Delta,r_2,\ii)}_2)|^2 ] \leq C_4^2  E[|X^{(\Delta,r_2,\ii)}_2-a_\ii|_2^{2(Q+1)}].
$$
Since,
$$
|X^{(\Delta,r_2,\ii)}_2-a_i|_2\leq \sqrt{D} \frac{h}{2}+ C_{b,\sigma}(\Delta+\sqrt{\Delta} [\xi]_{r_2}), 
$$
the term  
$$
E[|X^{(\Delta,r_2,\ii)}_2-a_\ii|_2^{2(|\iota|_1+\rho+1)}]
$$
is of the order $\Delta^{(2Q+2)\gamma_{cube}}=\Delta^{\rho+|\iota|_1}$. Thus, the proof of \eqref{eq:hilf008} is complete.
\end{proof}

\section{Outlook}

The general ideas behind the RAWBFST algorithm are not restricted to the choice of Legendre polynomials localized one time step ahead 
as basis functions. An error analysis based on Theorem \ref{thm:main} can, in principle, be carried out in the Euler scheme  setting for 
any set of
basis functions that depend on $(X_1,\xi)$. If we think of BSDE applications similar to the example in Section \ref{sec:appBSDE}, we 
can e.g. combine the generic basis choice of `regression now' with the variance reduction benefits of noiseless regression in the following way: Choose any $D+1$ sets of basis functions 
$$
\{\eta_{d,k}(x):k=1,\ldots K_d \}, \quad d=0,\ldots,D, 
$$
for a standard `regression now' approach to the approximation of the $(D+1)$ conditional expectations
\begin{equation}\label{eq:outlook}
E[y(X_2)|X_1=x],\quad E\left[\frac{\xi_d}{\sqrt{\Delta}} y(X_2)|X_1=x\right],\;d=1,\ldots, D.
\end{equation}
Instead one can run a single least-squares regression of $y(X_2)$ on the basis functions
\begin{eqnarray*}
&&\eta_{0,1}(X_1),\ldots \eta_{0,K_0}(X_1),\eta_{1,1}(X_1)\sqrt{\Delta} \xi_1,\ldots,  \eta_{1,K_1}(X_1)\sqrt{\Delta} \xi_1, \ldots, \\ && \;
\eta_{D,1}(X_1)\sqrt{\Delta} \xi_D,\ldots,  \eta_{1,K_D}(X_1)\sqrt{\Delta} \xi_D,
\end{eqnarray*}
for which the required conditional expectations (with the convention $\xi_0:=1$)
$$
E[\frac{\xi_d}{\sqrt{\Delta}} \eta_{j,k}(X_1)\sqrt{\Delta} \xi_j|X_1=x]= {\bf 1}_{\{d=j\}} \eta_{j,k}(x)
$$
are available in closed form. Hence, an analogue of RAWBFST with this set of basis functions finally approximates the conditional expectations 
in \eqref{eq:outlook} as linear combination of the same basis functions as the classical `regression now' does, but can potentially 
benefit from a similar automatic variance reduction as RAWBFST does.
While details are beyond the scope of the present paper, we note that this type of basis choice has been successfully implemented 
in the context of robust multiple stopping in the recent preprint \cite{Lal}.
This example 
illustrates the wide flexibility in the design of `regression anytime' algorithms for 
dynamic programming equations even compared to the classical `regression now' framework. 

\section*{Acknowledgements}
We thank Christian G\"artner for his contributions to unpublished predecessors of this project.

\end{document}